\documentclass{article}
\usepackage{amsmath, amsfonts, amsthm, amssymb, amscd}
\usepackage[mathcal]{euscript}
\usepackage{dsfont, pxfonts, mathrsfs, accents, verbatim}
\theoremstyle{plain}
        \newtheorem{theorem}{Theorem}[section]
\newtheorem{definition}[theorem]{Definition}
        \newtheorem{proposition}[theorem]{Proposition}
        \newtheorem{lemma}[theorem]{Lemma}
        \newtheorem{corollary}[theorem]{Corollary}
        \newtheorem{remark}[theorem]{Remark}

\numberwithin{equation}{section}
\numberwithin{equation}{section}

\newcommand \be     {\begin{equation}}
\newcommand \ee     {\end{equation}}

\newcommand \del        \partial
\newcommand \eps        \varepsilon
\newcommand \auth   \textsc

\begin{document}

\title{Complete classification of compact four-manifolds with positive isotropic curvature}
\author{Bing-Long Chen, Siu-Hung Tang,  Xi-Ping Zhu}\maketitle

\begin{abstract}
In this paper, we completely classify all compact 4-manifolds with
positive isotropic curvature. We show that they are diffeomorphic to
$\mathbb{S}^4,$ or $\mathbb{R}\mathbb{P}^4$ or quotients of
$\mathbb{S}^3\times \mathbb{R}$ by a cocompact fixed point free
subgroup of the isometry group of the standard metric of
$\mathbb{S}^3\times \mathbb{R}$ ,  or a connected sum of them.

\end{abstract}

\section{Introduction}
\label{IN-0}

Let $M$ be an n-dimensional Riemannian manifold. Recall that its
curvature operator at $p \in M$ is the self adjoint linear
endomorphism ${\mathcal{R}}: {\wedge^2} T_pM \to {\wedge^2} T_pM$
defined by
$$ <{\mathcal{R}}(X {\wedge} Y), U {\wedge} V> = <Rm(X,Y)V,U>,
\,\,\,\,\,\text{for} \,\,\,\, X,Y,U,V \in T_pM. $$ Here $<\, , \,>$
is the Riemannian metric and $Rm$ is the Riemann curvature tensor on
$M$. The Riemannian metric $<\, , \,>$ can be extended either to a
complex bilinear form $(\, , \,)$ or a Hermitian inner product $<<\,
, \,>>$ on $T_pM \otimes \mathbb{C}$. We extend the curvature
operator to a complex linear map on ${\wedge^2} T_pM \otimes
\mathbb{C}$, also denoted by ${\mathcal{R}}$. Then, to every two
plane $\sigma \subset T_pM \otimes \mathbb{C}$, we can define the
complex sectional curvature $K_{\mathbb{C}}(\sigma)$ by
$$K_{\mathbb{C}}(\sigma) = <<{\mathcal{R}}(Z \wedge W) , Z \wedge W>>$$
where $\{ Z, W \}$ is a unitary basis of $\sigma$ with respect to $
<<\, , \,>>$. We say that $M$ has positive isotropic curvature (PIC
for short) if $K_{\mathbb{C}}(\sigma) > 0$ whenever $\sigma \subset
T_pM \otimes \mathbb{C}$ is a totally isotropic two plane for any $p
\in M$. Here $\sigma$ is totally isotropic if $(Z,Z)= 0$ for any $Z
\in \sigma$. To clarify the meaning of positive isotropic curvature,
we have the following diagram for the relative strength of the
positivity for various notions of curvatures.
$$ \begin{array}{rcccc} {\mathcal{R}} > 0 \Rightarrow &
K_{\mathbb{C}} > 0 \Rightarrow & K > 0 \Rightarrow & Ric > 0
\Rightarrow & R > 0 \\ & \Downarrow & & & \\ \text{pointwise 1/4
pinching}\Rightarrow & \text{PIC}\Rightarrow & R
> 0 & &
\end{array} $$
Here, $K$ is the sectional curvature, i.e the restriction of
$K_{\mathbb{C}}$ on real 2 planes in $T_pM \otimes \mathbb{C}$, $Ric
$ is the Ricci curvature and $R$ is the scalar curvature on $M$. The
pointwise 1/4 pinching condition means that for any $p \in M$, we
have
$$ 1 < \frac{\max \{K(\sigma):\, \text{2 plane} \, \sigma \subset T_pM \} }
{\min \{ K(\sigma):\, \text{2 plane} \, \sigma \subset T_pM \} }
\leq 4. $$

The notion of positive isotropic curvature was introduced in the
paper of Micallef and Moore \cite{MiMo} in 1988 where they
discovered that it can be used to control the stability of minimal
surfaces just as the notion of positive sectional curvature can be
used to control the stability of geodesics. Hence by using minimal
surface theory, they proved

\vspace{0.2cm}

\textbf{Theorem (Micallef-Moore).} \textit{ Let $M$ be a compact
simply connected n-dimensional manifold with positive isotropic
curvature where $n \geq 4$. Then $M$ is homeomorphic to a sphere.}

\vspace{0.2cm}

In view of above diagram, for $n \geq 4$, if $M$ is a compact simply
connected n-dimensional manifold with positive curvature operator or
pointwise 1/4 pinching, then $M$ is homeomorphic to a sphere. The
latter generalizes the famous sphere theorem of Berger and
Klingenberg. It is spectacular that, by using the Ricci flow, it was
proved recently in B\"{o}hm-Wilking \cite{BW} and Brendle-Schoen
\cite{BS1} that a compact n-dimensional simply connected manifold
with positive curvature operator or pointwise 1/4 pinching is indeed
diffeomorphic to the round sphere $S^n$.

In 1997, in a seminal paper \cite{Ha97}, Hamilton initiated the
study of positive isotropic curvature by Ricci flow. In dimension 4,
he first proved that the condition of positive isotropic curvature
is preserved under Ricci flow. Then, under the assumption that there
is no essential incompressible space forms in the manifold, he
developed a theory of Ricci flow with surgery to exploit the
development of singularities in the Ricci flow to recover the
topology of the manifold. Here an incompressible space form $N$ in a
four manifold $M$ is a smooth submanifold diffeomorphic to a
spherical space form $S^3/\Gamma$ such that the inclusion induces an
injection from $\pi_1(N)$ to $\pi_1(M)$. It is essential unless
$\Gamma = 1$ or $\Gamma = {\mathbb{Z}}_2$ and the normal bundle is
unorientable. Hamilton's paper contained some unjustified statements
which were later supplemented by the paper of Chen and Zhu
\cite{CZ05F}. Their main result is

\vspace{0.2cm}

\textbf{Theorem (Hamilton).} \textit{ Let $M$ be a compact four
manifold with no essential incompressible space form. Then $M$
admits a metric with positive isotropic curvature if and only if
it is diffeomorphic to ${\mathbb{S}}^4, {\mathbb{R}}
{\mathbb{P}}^4, {{\mathbb{S}}^3} \times {{\mathbb{S}}^1},
{{\mathbb{S}}^3} \widetilde{\times} {{\mathbb{S}}^1}$(this is the
quotient of ${{\mathbb{S}}^3} \times {{\mathbb{S}}^1}$ by
${\mathbb{Z}}_2$ which acts by reflection and  antipodal map on
the first and second factor respectively), or a connected sum of
them.}

\vspace{0.2cm}

Clearly, each of the manifolds ${\mathbb{S}}^4, {\mathbb{R}}
{\mathbb{P}}^4, {{\mathbb{S}}^3} \times {{\mathbb{S}}^1},
{{\mathbb{S}}^3} \widetilde{\times} {{\mathbb{S}}^1}$ listed in the
above theorem admits a metric with positive isotropic curvature. A
theorem of Micallef and Wang \cite{MW} guarantees that the connected
sum of compact manifolds with positive isotropic curvature also
admits such a metric. Another useful observation is that the
condition of no essential incompressible space form is automatically
satisfied if $\pi_1(M)$ is torsion free, i.e. contains no nontrivial
element of finite order. Indeed, $\Gamma$ in the above definition of
essential incompressible space form must be trivial. So, if the
fundamental group of a compact Riemannian four manifold $M$ with
positive isotropic curvature contains a normal torsion free subgroup
of finite index, then a finite cover of $M$ is diffeomorphic to
${\mathbb{S}}^4, {{\mathbb{S}}^3} \times {{\mathbb{S}}^1}$ or a
connected sum of them. This shows the intimate connection between
the topology and the fundamental group of a compact Riemannian
manifold with positive isotropic curvature, at least in dimension 4.

For dimension greater than 4, it has been proved recently by Brendle
and Schoen \cite{BS1} that the condition of positive isotropic
curvature is preserved under Ricci flow although there is yet no
generalization of the curvature pinching estimates which is crucial
in Hamilton's analysis of \cite{Ha97}. Another interesting result
for higher dimensional Riemannian manifold with positive isotropic
curvature is the result of  Fraser and Wolfson \cite{Fr} \cite {FW}
who proved that the fundamental group of any compact surface of
genus $g \geq 1$ cannot occur as a subgroup of such manifold when
its dimension is greater than 4.

Recently, Schoen \cite{S} proposed the following

\vspace{0.2cm}

\textbf{Conjecture (Schoen).} \textit{For $n \geq 4$, let $M$ be an
n-dimensional compact Riemannian manifold with positive isotropic
curvature. Then a finite cover of $M$ is diffeomorphic to
${\mathbb{S}}^n, {{\mathbb{S}}^{n-1}} \times {{\mathbb{S}}^1}$ or a
connected sum of them. In particular, the fundamental group of $M$
is virtually free.}

\vspace{0.2cm}

The purpose of this paper is to prove the conjecture of Schoen when
$n=4$. Indeed, we obtain a more precise result. In particular, we
know exactly what are the fundamental groups of such manifolds. Our
main result is

\vspace{0.2cm}

\textbf{Main Theorem.} \textit{Let $M$ be a compact 4-dimensional
manifold. Then it admits a metric with
 positive isotropic curvature if and only if it is diffeomorphic to
$\mathbb{S}^4$, $\mathbb{R}\mathbb{P}^4$, $\mathbb{S}^3\times
\mathbb{R}/G$ or a connected sum of them. Here $G$ is a cocompact
fixed point free discrete subgroup of the isometry group of the
standard metric on ${{\mathbb{S}}^3} \times {{\mathbb{R}}}$.}

\vspace{0.2cm}
 We give two immediate corollaries of our Main
Theorem.

\vspace{0.2cm}

\textbf{Corollary 1.} \textit{The conjecture of Schoen is true for
$n=4$.}

\begin{proof}
There is nothing to prove if $M$ is diffeomorphic to
$\mathbb{S}^4$ or $\mathbb{R}\mathbb{P}^4$. So, we may assume that
$M$ is diffeomorphic to $ m \mathbb{R}\mathbb{P}^4 \#
\mathbb{S}^3\times \mathbb{R}/G_1\#\cdots\# \mathbb{S}^3\times
\mathbb{R}/G_k$ for some nonnegative integer $m$ and positive
integer $k$. The fundamental group of $M$ is given by
$$\underbrace{{\mathbb{Z}}_2
\ast \cdots \ast {\mathbb{Z}}_2}_{m \,\text{times}} \ast G_1 \ast
\cdots \ast G_k.$$ Now a cocompact fixed point free discrete
subgroup $G$ of the isometry group of ${{\mathbb{S}}^3} \times
{{\mathbb{R}}}$ is always virtually infinite cyclic. This is
because, by the cocompactness of the action of $G$ on
${{\mathbb{S}}^3} \times {{\mathbb{R}}}$, $G$ always contains an
element $g$ which acts as translation on the second factor and the
infinite cyclic subgroup generated by $g$ must have finite index
as it also acts cocompactly on ${{\mathbb{S}}^3} \times
{{\mathbb{R}}}$. Thus $\pi_1(M)$ is the free products of finite
and virtually infinite cyclic groups. It is known that such group
always contains a normal free subgroup of finite index. In
particular, $\pi_1(M)$ contains a torsion free normal subgroup of
finite index. By the remark after the statement of Hamilton's
Theorem, the conclusion in the conjecture of Schoen holds.
\end{proof}

The second corollary concerns the classification of compact
conformally flat Riemannian four manifolds with positive scalar
curvature. We start with a digression of the geometry of Riemannian
four manifold $M$. In this case, the bundle ${\wedge^2} TM$ has a
decomposition into the direct sum of its self-dual and
anti-self-dual parts
$$ {\wedge^2} TM = {\wedge^2_+} TM \oplus  {\wedge^2_-} TM. $$
The curvature operator can then be decomposed as
$${\mathcal{R}}=\left(\begin{array}{cc}
A&B\\ B^t&C\end{array}\right)$$ where $A = W_+ + \frac{R}{12}$, $B
=\stackrel{\circ}{Ric}$, $C = W_- + \frac{R}{12}$. Here $W_{\pm}$
are the self-dual and anti-self-dual Weyl curvature tensors
respectively while $\stackrel{\circ}{Ric}$ is the trace free part of
the Ricci curvature tensor. Denote the eigenvalues of the matrices
$A,$ $C$ and $\sqrt{BB^{t}}$ by $a_1\leq a_2 \leq
 a_3,$  $c_1\leq c_2 \leq
 c_3,$ $b_1\leq b_2 \leq
 b_3$
 respectively. It is known that the condition of positive isotropic curvature is
 equivalent to the conditions $a_1+a_2>0$ and $c_1+c_2>0.$ From this, it is clear
 that a compact
conformally flat Riemannian four manifold with positive scalar
curvature always has positive isotropic curvature.

Now it had been observed by Izeki \cite{Iz} that a compact
conformally flat Riemannian four manifold $M$ with positive scalar
curvature always has a finite cover which is diffeomorphic to
${\mathbb{S}}^4, {{\mathbb{S}}^{3}} \times {{\mathbb{S}}^1}$ or a
connected sum of them. The reason is this. Let $M$ be  such a
manifold, then by a result of Schoen and Yau \cite{SY}, $\pi_1(M)$
is a Kleinian group. In particular, it is a finitely generated
subgroup of a linear group, namely $SO(5,1)$. By Selberg's Lemma,
$\pi_1(M)$ contains a torsion free normal subgroup of finite index.
Since such manifold always has positive isotropic curvature, we can
again apply the above remark after the statement of Hamilton's
Theorem to conclude that $M$ has a finite cover which is
diffeomorphic to ${\mathbb{S}}^4, {{\mathbb{S}}^{3}} \times
{{\mathbb{S}}^1}$ or a connected sum of them.

Our Main Theorem gives a more precise classification of such
manifolds.

\vspace{0.2cm}

\textbf{Corollary 2.} \textit{A compact  four manifold admits a
metric of positive isotropic curvature if and only if it admits a
conformally flat metric of positive scalar curvature.}

\begin{proof} The manifolds $\mathbb{S}^4$, $\mathbb{R}\mathbb{P}^4$, $\mathbb{S}^3\times
\mathbb{R}/G$ listed in the Main Theorem clearly admit conformally
flat metrics of positive scalar curvature and we only have to invoke
the fact that connected sum of conformally flat Riemannian manifolds
with positive scalar curvature also admits such a metric.
\end{proof}

We remark that corollary 2 does not hold for dimension
  $n>4.$ The following example is taken  from \cite{MW}.  For any Riemann surface $\Sigma_g$ of genus $g\geq 2$
  and $n>4$,
  the manifold $M = \Sigma_g\times \mathbb{S}^{n-2}$ admits
  a conformally flat metric of positive scalar curvature, however,
  because of the above mentioned result of Fraser and Wolfson \cite{FW}, $M$
  cannot admit a metric with positive isotropic curvature.


The proof of our Main Theorem naturally divides into two parts. The
first part is analytical and the second part topological.

Our argument in the first part is based on the celebrated
Hamilton-Perelman theory \cite{Ha97} \cite{P2} on the the Ricci flow
with surgery. To approach the topology of a compact four-manifold
with positive isotropic curvature, we take it as initial data and
evolve it by the Ricci flow. It is easy to see that the solution
will blow up in finite time. By applying Hamilton's curvature
pinching estimates obtained in \cite{Ha97}, we can get a complete
understanding on the part around the singularities of the solution.
Then we can perform Hamilton's surgery procedure to cutoff the part
around the singularities. After the surgery, due to the possible
existence of essential incompressible space forms, we will get a
closed (maybe not connected) orbifold with positive isotropic
curvature. After studying Ricci flow on orbifold and obtaining a
detailed singularity analysis for orbifold Ricci flow, we can use
the orbifold as initial data to run the Ricci flow and to do
surgeries again. By repeating this procedure and extending the
arguments in the previous paper \cite{CZ05F} of the first and the
third authors to the orbifold case, we will be able to show that,
after a finite number of surgeries and discarding a finite number of
pieces which are diffeomorphic to spherical orbifolds
$\mathbb{S}^4/\Delta$ (here $\Delta$ denotes a finite subgroup of
the orthogonal group $O(5)$) with at most isolated orbifold
singularities, the solution becomes extinct. As a result, we prove
that the initial manifold is diffeomorphic to an orbifold connected
sum (see below or the precise definition given in section 2) of
spherical orbifolds $\mathbb{S}^4/\Delta.$

The second part concerns the recovery of the topology of the
manifold from the orbifold connected sum. First of all, by an
algebraic lemma, we know that a spherical orbifold
$\mathbb{S}^4/\Delta$ has either zero, one or two orbifold
singularities. A spherical orbifold with no orbifold singularity
is simply $\mathbb{S}^4$ or $\mathbb{R}\mathbb{P}^4$ while those
with one or two orbifold singularities is, after removing an open
neighborhood from each of its orbifold singularities,
diffeomorphic to a smooth cap or a cylinder respectively. Here a
cylinder $C(\Gamma)$ is given by $\mathbb{S}^3/\Gamma \times
[-1,1]$ for some finite fixed point free subgroup $\Gamma$ of
$SO(4)$ while a smooth cap $C_{\Gamma}^{\sigma}$ is given as the
quotient of $\mathbb{S}^3/\Gamma \times [-1,1]$ by a group of
order two generated by $\hat{\sigma}: (x,s) \mapsto
(\sigma(x),-s)$ where $\sigma$ is a fixed point free isometric
involution on $\mathbb{S}^3/\Gamma$. Now, the orbifold connected
sum of spherical orbifolds is formed in two steps. In the first
step, to undo the surgeries in the Ricci flow which create
orbifold singularities, we glue copies of the $C(\Gamma)$'s and
$C_{\Gamma}^{\sigma}$'s along their diffeomorphic boundaries with
suitable identifying maps to form a number of closed (compact)
manifolds. It is not hard to see that, up to diffeomorphisms, they
are essentially of two types: the self-gluing of the two ends of a
cylinder $C(\Gamma)$ and the gluing of two smooth caps,
$C_{\Gamma}^{\sigma}$ and $C_{\Gamma}^{\sigma'}$, with
diffeomorphic boundaries by suitable diffeomorphisms on
$\mathbb{S}^3/\Gamma$. Since we know that any diffeomorphism on a
three dimensional spherical space form is isotopic to an isometry.
The resulting closed manifolds can be equipped with metrics which
are locally isometric to $\mathbb{S}^3 \times \mathbb{R}$. Now,
the second step in the formation of the orbifold connected sum
consists of two types of operations. The first is the usual
connected sum of the above closed manifolds with $\mathbb{S}^4$'s
and $\mathbb{R}\mathbb{P}^4$'s and the second is adding handles to
them. Since the latter operation is in term equivalent to the
connect sum of them with $\mathbb{S}^3 \times \mathbb{S}^1$ or
${{\mathbb{S}}^3} \widetilde{\times} {{\mathbb{S}}^1}$, our Main
Theorem is proved.

A natural question is whether our Main Theorem and its proof can be
extended to dimension greater than 4. We believe that the analytic
part of our proof will go through once Hamilton's curvature pinching
estimates in \cite{Ha97} can be extended to higher dimensions.
Assuming that this has been done, most of the argument in the
topological part of our proof will also go through. This will allow
us to show that a compact Riemannian n-dimensional manifold $M$ with
positive isotropic curvature is homeomorphic to $\mathbb{S}^n$,
$\mathbb{R}\mathbb{P}^n$, $\mathbb{S}^{n-1}\times \mathbb{R}/G$ or a
connected sum of them. Here we only know that $G$ acts
differentiably on ${{\mathbb{S}}^{n-1}} \times {{\mathbb{R}}}$.
The differences are due to the possible existence of (exotic)
diffeomorphisms on a spherical space form
${\mathbb{S}}^{n-1}/\Gamma$ which is not isotopic to an isometry. By
the same argument as in our proof of Corollary 1, this result still
implies a weaker form of the conjecture of Schoen, namely, $M$ has a
finite cover which is homeomorphic to $\mathbb{S}^n$,
$\mathbb{S}^{n-1}\times \mathbb{S}^1$ or a connected sum of them.

Our paper is organized as follows. In section 2, we introduce some
terminologies and state one of the main results of the paper,
Theorem \ref{t2}, which says that any 4-orbifold with positive
isotropic curvature and with at most isolated singularities is
diffeomorphic to an orbifold connected sum of spherical orbifolds
$\mathbb{S}^4/\Gamma.$ In section 5, we identify these orbifold
connected sums and prove the Main Theorem. The proof of Theorem
\ref{t2}, by Ricci flow, will occupy sections 3 and 4. Section 6
gives the proof of a geometric lemma which is used frequently in the
paper.

{\bf Acknowledgements} The first and the third author are
partially supported by NSFC 10831008 and
 NKBRPC 2006CB805905. The second author is partially supported by
 NSFC 10831008.

\section{Orbifold connected sum }
\label{IN-0} \label{2.1}

We generalize the construction of connected sum of manifolds to
orbifolds with at most isolated singularities. For an orbifold $X,$
$x\in X,$ we use $\Gamma_{x}$ to denote the local uniformization
group at $x,$ namely, there is a open neighborhood $B_x \ni x$ with
smooth boundary which is a quotient $\tilde{B}/\Gamma_x,$ where
$\tilde{B}$ is diffeomorphic to  $ \mathbb{R}^n$ and $\Gamma_{x}$ is
a finite subgroup of linear transformations fixing the origin. Let
$X_1,\cdots, X_p$ be $n-$dimensional orbifolds with at most isolated
orbifold singularities.  Let $x_1,x_1',x_2,x_2',\cdots,x_q' $ be
$2q$ distinct points (not necessarily singular) on $X_1,\cdots, X_p$
such that for each pair $(x_j,x_j'),$ $\Gamma_{x_{j}}$ is conjugate
to $\Gamma_{x_{j}'}$ as linear subgroups. Assume $x_j\in X_{i_{j}}$
and $x_j'\in X_{i_{j}'}$ for $j=1,\cdots,q.$ Let $f_j$ be  a
diffeomorphism from $\partial B_{x_j}$ and $\partial B_{x_j^{'}}$
for $j=1,2,\cdots, q. $ For each $j,$ we remove $x_{j}$ and $x_{j}'$
from the orbifolds, and identify the boundary $\partial B_{x_j}$
with $\partial B_{x_j^{'}}$ by using the diffeomorphism $f_j.$  Let
$f=(f_1,\cdots,f_q).$ We denote the resulting space by
${\#}_{f}(X_1,\cdots, X_p).$ We call it an orbifold connected sum of
$X_1,\cdots, X_p.$ Here we emphasis that the diffeomorphism type of
the resulting orbifold  depends only on the isotopic class of $f.$
 Now we specify our construction to dimension 4.

 One of the main efforts of this paper is to show the
  following:

 \begin{theorem}\label{t2}Let $(M^4,g)$ be a compact 4-dimensional manifold
 or  orbifold with at most isolated singularities with
 positive isotropic curvature.  Then $M^4$ is diffeomorphic to
 an orbifold connected sum  of  a finite
 number of spherical 4-orbifolds $X_1=\mathbb{S}^4/\Gamma_1,\cdots, X_l=\mathbb{S}^4/\Gamma_l,$
 where each $\Gamma_i$ is a finite subgroup of the isometry group, $O(5)$, of the standard metric
 on $\mathbb{S}^4$ so that
 the quotient orbifold $X_i$ has at most isolated singularities.
  \end{theorem}

  Now, we discuss some natural examples of compact four manifolds
  with positive isotropic curvature.  We will describe their constructions
from  orbifold connected sums by spherical orbifolds.

 In dimension 4, except for $\mathbb{S}^4$ and $\mathbb{R}\mathbb{P}^4,$
 the best known examples of positive isotropic curvature are
$\mathbb{S}^3/ \Gamma\times \mathbb{S}^1,$ where $\Gamma$ is a fixed
point free finite subgroup of $SO(4)$. Clearly $\Gamma$ can also act
isometrically on $\mathbb{S}^4$
 by fixing an axis. The
 orbifold $\mathbb{S}^4/ \Gamma$ has exactly two singulaities
 $P$ and $P'.$ Clearly, if one performs an orbifold connected sum on $\mathbb{S}^4/ \Gamma$ with itself
 by using the identity map as the identifying map, it gives $\mathbb{S}^3/ \Gamma\times
 \mathbb{S}^1.$  If we choose the identifying map $f$ ( in Diff$(\mathbb{S}^3/\Gamma)$)
 in a nontrivial isotopic class, then the connected sum may give some twisted product
of $\mathbb{S}^3/ \Gamma$ and $ \mathbb{S}^1.$  We denote the
 manifold by
$\mathbb{S}^3/ \Gamma{\times}_{f}
 \mathbb{S}^1.$ By \cite{Mc},  the mapping class group of three
 dimensional
spherical space form $\mathbb{S}^3/\Gamma$ is a finite group. So for
each $\Gamma,$ there is only a finite number of diffeomorphism
classes of $\mathbb{S}^3/ \Gamma{\times}_{f}
 \mathbb{S}^1.$ In particular, when $\Gamma=\{1\}$ and $f$ is an
 orientation reversing diffeomorphism, the resulting manifold is
 $\mathbb{S}^3\tilde{\times}\mathbb{S}^1,$ which is the only
 unoriented $\mathbb{S}^3$ bundle over $\mathbb{S}^1.$

    If
 $\mathbb{S}^3/ \Gamma$ admits  a fixed point free
 isometry $\sigma$ satisfying $\sigma^2=1,$ then we can define a reflection $\hat{\sigma}$ on the 4-manifold
 $\mathbb{S}^3/ \Gamma\times \mathbb{R}$ by $\hat{\sigma}(x,s)=(\sigma(x),-s),$ where $x\in \mathbb{S}^3/ \Gamma,
 s\in \mathbb{R}. $
 The quotient $(\mathbb{S}^3/ \Gamma\times \mathbb{R})/\{1,\hat{\sigma}\}$ is a
 smooth four manifold  with neck like end $\mathbb{S}^3/ \Gamma \times
 \mathbb{R}.$ We denote the manifold by
 $C_{\Gamma}^{\sigma}.$ If we think of the sphere  $\mathbb{S}^4$ as
 the compactification of $\mathbb{S}^3/ \Gamma\times
 \mathbb{R}$ by adding two points (north and south poles) at
 infinities of $\mathbb{S}^3/ \Gamma\times
 \mathbb{R}$, we can regard $\Gamma$ and
 $\hat{\sigma}$ as isometries of the standard $\mathbb{S}^4$ in a natural manner. So
 $C_{\Gamma}^{\sigma}$ is diffeomorphic to the smooth manifold
 obtained by removing the unique singularity from $\mathbb{S}^4/\{\Gamma,
 \hat{\sigma}\}.$ We call $C_{\Gamma}^{\sigma}$ smooth cap.

 Given two smooth caps $C_{\Gamma}^{\sigma}$ and $C_{\Gamma'}^{\sigma'},$ if $\Gamma$ is conjugate to
 $\Gamma'$(i.e. there is an  isometry $\gamma$ of $\mathbb{S}^3$ such that $\Gamma=\gamma \Gamma'
 \gamma^{-1}$), we can glue $C_{\Gamma}^{\sigma}$ and $C_{\Gamma'}^{\sigma'}$
  along their boundaries by a diffeomorphism $f:\partial C_{\Gamma}^{\sigma}\rightarrow \partial C_{\Gamma'}^{\sigma'}.$
 Then  we get a smooth manifold and we denote it by
 $C_{\Gamma}^{\sigma}\cup_{f}
  C_{\Gamma'}^{\sigma'}.$   Let  $P,P'$ be the
 singularities of the orbifolds $\mathbb{S}^4/ \{\Gamma,\hat{\sigma}\}$ and
 $\mathbb{S}^4/ \{\Gamma',\hat{\sigma'}\}$ If we resolve these two singularities by orbifold connected sum
 with some diffeomorphism $f$ between the boundaries of a neighborhood of the singular points,
 we get
 $C_{\Gamma}^{\sigma}\cup_f
  C_{\Gamma'}^{\sigma'}$.
A simple example for $\Gamma=\Gamma'=\{1\}$ is
$\mathbb{R}\mathbb{P}^4\#\mathbb{R}\mathbb{P}^4,$ which is a
quotient of ${{\mathbb{S}}^3} \times {{\mathbb{S}}^1}$ by
${\mathbb{Z}}_2$ which acts by  antipodal map and  reflection  on
the first and second factor respectively.

   The proof of theorem \ref{t2} will occupy
sections 3, 4. The method is to use Ricci flow to deform the initial
metric. By developing singularities, Ricci flow allows us to find
the necks connecting these spherical orbifolds. We disconnect these
spherical orbifolds by cutting off the necks between them. Let us
start to consider Ricci flow.

Let $(M^4,g_0)$ be a compact 4-dimensional  orbifold with at most
isolated singularities with
 positive isotropic curvature. We deform the initial metric by the Ricci flow equation:
\begin{equation}\label{3.1}
 \frac{\partial
g}{\partial t}=-2Ric, \ \ \ \ \  g\mid_{t=0}=g_0.
\end{equation}
Since the implicit function theorem or De Turck trick can also be
applied on orbifolds, we have the  short time solution $g(\cdot,t)$
of (\ref{3.1}) (see \cite{Ha}, \cite{Ha82}, \cite{De}). Recall that
as in the introduction, in dimension 4, the curvature operator has
the following decomposition
$${\mathcal{R}}=\left(\begin{array}{cc}
A&B\\ B^t&C\end{array}\right)$$ and we denote the eigenvalues of
matrices $A,$ $C$ and $\sqrt{BB^{t}}$ by $a_1\leq a_2 \leq
 a_3,$  $c_1\leq c_2 \leq
 c_3,$ $b_1\leq b_2 \leq
 b_3$
 respectively.
Since the maximum principle can also be applied on orbifolds, the
positivity of isotropic curvature and improved pinching estimates of
Hamilton are also preserved under the Ricci flow. We have

\begin{lemma} (Theorem B1.1 and Theorem B2.3 of \cite{Ha97})

There exist positive constants $\rho$, $\Lambda,P<+\infty$
depending only on the initial metric, such that the solution to
the Ricci flow (\ref{3.1}) satisfies
\begin{equation}\begin{split}\label{pinch} & a_1+\rho>0
\mbox{ and }c_1+\rho>0, \\
& \max\{a_3,b_3,c_3\}\leq \Lambda (a_1+\rho)
\max\{a_3,b_3,c_3\}\leq
\Lambda (c_1+\rho),\\
& \ \ \ \frac{b_3}{\sqrt{(a_1+\rho)(c_1+\rho)}}\leq
1+\frac{\Lambda e^{Pt}}{\max\{\log\sqrt{(a_1+\rho)(c_1+\rho)},2\}}
\end{split}
\end{equation}
\end{lemma}

So any blowing up limit satisfies  the following restricted
isotropic curvature pinching condition
\begin{equation}\label{ricp}a_3\leq \Lambda a_1,\ \ c_3\leq
\Lambda c_1,\ \ b_3^2\leq a_1c_1. \end{equation}

We can also define the same notion of $\kappa$ non-collapsed for a
scale $r_0$ for solutions to the Ricci flow on orbifolds, namely,
for any space time point $(x_0,t_0)$, the condition that
$|Rm|(x,t)\leq r_0^{-2},$ $\mbox{for all } t\in [t_0-r_0^2,t_0]$ and
$ x\in B_t(x_0,r_0),$ implies $Vol_{t_0}(B_{t_0}(x_0,r_0))\geq
\kappa r_0^4.$ Since integration by parts and log-Sobolev inequality
still hold on closed orbifolds, we can apply the same argument as in
\cite{P1} (Theorem 4.1 of \cite{P1} or see Lemma 2.6.1 and Theorem
3.3.3 of \cite{CaoZ} for the details) to show

\begin{lemma} \label{ka}For any $T>0,$ there is a $\kappa$ depending on $T$
and the initial orbifold metric, such that the smooth solution to
the Ricci flow which exists for $[0,T)$ is $\kappa$ non-collapsed
for scales less than $\sqrt{T}.$
\end{lemma}

Since the scalar curvature is strictly positive, it follow from the
standard maximum principle and the evolution equation of the scalar
curvature that the solution must blow up at finite time. As in the
smooth case \cite{CZ05F}, we will show that the geometric structure
at any point with suitably large curvature is close to an ancient
$\kappa-$solution. So it is important to investigate the structures
of any ancient $\kappa-$solutions. This is done in section
\ref{ancient}.

For the convenience of discussion, we need to fix some terminologies
and notations.

In this paper, a (topological) neck is defined to be diffeomorphic
to $\mathbb{S}^3/ \Gamma\times \mathbb{R}.$ Here $\Gamma$ is a
finite fixed point free subgroup of isometries of $\mathbb{S}^3.$
For caps, we define smooth caps consisting of
$C_{\Gamma}^{\sigma}$ and $\mathbb{B}^4$and we define two types of
orbifold caps. The orbifold cap of Type I is obtained by crunching
the boundary $\mathbb{S}^3/ \Gamma\times
 \{0\}$ of $\mathbb{S}^3/ \Gamma\times
 [0,1)$ to a point. We denote it by $C_{\Gamma}.$ By extending the
 action $\Gamma$ to isometric actions of $\mathbb{S}^4,$ it is
 clear  $C_{\Gamma}$ is obtained by removing one singularity from the spherical orbifold  $\mathbb{S}^4/\Gamma.$
To define the orbifold cap of type II, we
first construct certain spherical orbifold in the following manner.
We write the equation of $\mathbb{S}^4$ as ${x_1}^2+\cdots
+{x_5}^2=1,$ then the isometry $(x_1,x_2\cdots, x_5)\rightarrow
(x_1,-x_2\cdots,-x_5)$ has exactly two fixed points $(1,0,0,0,0)$
and $(-1,0,0,0,0)$ with local uniformization group $\mathbb{Z}_2.$
We denote this spherical orbifold by $\mathbb{S}^4/(x,\pm x').$ The
orbifold cap of Type II, denoted  by $\mathbb{S}^4/(x,\pm
x')\backslash
 \bar{\mathbb{B}}^4,$ is obtained by removing a smooth
point from
 spherical orbifold $\mathbb{S}^4/(x,\pm x').$

Roughly speaking, we will show in section \ref{ancient} that either
the ancient $\kappa-$solution is diffeomorphic to a global quotient
$\mathbb{S}^4/\Gamma$ or else it has local structures of necks,
smooth
caps, or orbifold caps of type I or II described in the above. 

  \section{Ancient $\kappa-$solutions on orbifolds}
\label{ancient}

\begin{definition}\label{def1}We say a solution to the Ricci flow  is an
 \textbf{ancient $\kappa-$orbifold solution}  if it is a smooth complete
 nonflat
solution to the Ricci flow on a  four-orbifold with at most
isolated singularities satisfying the following three conditions:

(i) the solution exists on the ancient time interval
$t\in(-\infty,0],$ and

(ii) it has positive isotropic curvature and bounded curvature,
and satisfies the restricted isotropic curvature pinching
condition , \begin{equation}\label{restrictive}a_3\leq \Lambda
a_1,\ \ c_3\leq \Lambda c_1,\ \ b_3^2\leq a_1c_1,\end{equation}

(iii) $\kappa$-noncollapsed on all scales for some $\kappa>0.$
\end{definition} The purpose of this section is to describe the canonical neighborhood structure of
ancient $\kappa-$orbifold solutions.

\subsection{Curvature has null eigenvector}\label{sec3.2.1}


\begin{theorem}\label{t3.4}
Let $(X,g_t)$ be an ancient $\kappa-$orbifold solution defined in
Definition \ref{def1} such that the curvature operator has
nontrivial null eigenvector somewhere. Then we have

(i) if $X$ is smooth manifold, then either $X=
(\mathbb{S}^3/\Gamma)\times \mathbb{R},$ or $X=
C_{\Gamma'}^{\sigma}$ for some fixed point free isometric subgroup
$\Gamma$ or $\Gamma'$ of $\mathbb{S}^3$ and $\sigma$ is an fixed
point free isometry on $\mathbb{S}^3/\Gamma'$ with $\sigma^2 = 1$;

(ii) if $X$ has singularities, then $X$ is diffeomorphic to
$\mathbb{S}^4/(x,\pm x')\setminus \bar{\mathbb{B}}.$ In
particular, $X$ has exactly two singularities.
\end{theorem}
\begin{proof} Suppose the curvature operator
has nontrivial null eigenvector somewhere. Then the null
eigenvectors exist everywhere in space time by Hamilton's strong
maximum principle \cite{Ha86}.

Case 1:  $X$ is a smooth manifold.

In this case, it is known from Lemma 3.2 in \cite{CZ05F} that the
universal cover of $X$ is $\mathbb{S}^3\times \mathbb{R}.$ Let
$\Gamma$ be the group of deck transformations.
 We claim that the second components (acting on $\mathbb{R}$ isometrically) of $\Gamma$ must
contain no translations. Otherwise $X$ is compact. Note that the
flat $\mathbb{R}$ factor does not move during the Ricci flow, and
the spherical factor becomes very large when time goes to $-\infty$.
This contradicts with the $\kappa-$noncollapsing assumption. Let
$\Gamma=\Gamma^{0}\cup \Gamma^{1} $ where the second components of
$\Gamma^0$ and $\Gamma^1$ act on $\mathbb{R}$ as an identity or
reflection respectively. If $\Gamma^1$ is empty, $X=
(\mathbb{S}^3/\Gamma)\times \mathbb{R},$   where $\Gamma$ acts on
 $\mathbb{S}^3$ isometrically and has no fixed point. If $\Gamma^1$ is not empty, by
 picking
$\sigma\in \Gamma_{x}^1,$ then it satisfies $\sigma^2\in \Gamma^{0}$
and $\sigma \Gamma^{0}=\Gamma^1.$ It is clear that $X$ is obtained
by taking quotient of $ (\mathbb{S}^3/\Gamma^{0})\times \mathbb{R}$
by $\sigma.$ Hence $X=C_{\Gamma^{0}}^{\sigma}$ by using our notation
in section \ref{2.1}.

We remark that the $(\mathbb{S}^3/\Gamma)\times \mathbb{R}$ has
two ends, but $C_{\Gamma^{0}}^{\sigma}$ has only one end.

Case 2: $X$ is an orbifold with nonempty isolated singularities.

Since $X$ has local geometry of model $\mathbb{S}^3\times
\mathbb{R},$  $X$ must be a global quotient of $\mathbb{S}^3\times
\mathbb{R}$ by \cite{Th}, namely, $X=\mathbb{S}^3\times
\mathbb{R}/\Gamma,$ where $\Gamma$ is a subgroup of standard
isometries of $\mathbb{S}^3\times \mathbb{R}.$

Note that the fixed points of $\Gamma$ are isolated. For fixed
point $z\in  \mathbb{S}^3\times \mathbb{R},$  denote
$\Gamma_z=\{\gamma(z)=z,\gamma\in \Gamma\}.$ Let $\Gamma_0$ be the
minimal subgroup of $\Gamma$ containing all $\Gamma_{z}.$ Then
$\Gamma_0$ is a normal subgroup of $\Gamma.$ We claim the action
of  $G=\Gamma/\Gamma_0$ on $\mathbb{S}^3\times
\mathbb{R}/\Gamma_0$ has no fixed point. Indeed, if there are some
$g\in \Gamma$ and $x\in  \mathbb{S}^3\times \mathbb{R}$ such that
$g\Gamma_0(x)=\Gamma_0(x),$ this will imply $gx=\gamma x$ for some
$\gamma\in \Gamma_0.$ Hence $\gamma^{-1}g\in \Gamma_x\subset
\Gamma_0$ and $g\in \Gamma_0.$

 Pick  $\Gamma_{z}\neq \{1\}.$ Let
$\Gamma_{z}=\Gamma_{z}^0\cup \Gamma_{z}^1$ where the $\mathbb{R}$
components of $\Gamma_{z}^0$ and $\Gamma_{z}^1$ act on
$\mathbb{R}$ as an identity or reflection separately. We assume
 $z=(0,o)$ where $0\in \mathbb{R}$
and $o\in \mathbb{S}^3.$ Since $(0,o)$ is the unique fixed point
of each $\gamma\in \Gamma_{z},$ this implies
$\Gamma_{z}^0=\{1\}$(otherwise a nontrivial element of
$\Gamma_{z}^0$ will fix the whole $\{o\}\times \mathbb{R} $), and
$\Gamma_{z}^{1}=\{\sigma_{z}\},$  where the $\mathbb{S}^3$
component of $\sigma_{z}$ acts antipodally on the geodesic spheres
(isometric to scalings $\mathbb{S}^2$) of $\mathbb{S}^3$ at $o,$
since $\sigma_{z}$ has no fixed point on the geodesic spheres.

Note the $\mathbb{R}$ components of $\Gamma$ must contain no
translations. If we would have  two elements in $\Gamma$ whose
$\mathbb{R}$ components reflecting around points with different
$\mathbb{R}$ coordinates, then this will produce an element in
$\Gamma$ with nontrivial translation on  $\mathbb{R}$ factor. This
particularly implies that the all fixed points of $\Gamma$ have
same $\mathbb{R}$ coordinates. We may assume these fixed points
lie in $ \mathbb{S}^3\times \{0\}.$ We denote their reflections by
$\sigma_{z},\sigma_{w},\cdots.$ We can associate an equator (which
is the unique invariant equator)  to a $\sigma_z$ in an obvious
way.  Note the action $\sigma_{z}\sigma_{w}$ on $\mathbb{R}$ is
trivial. If $\sigma_{x}\neq\sigma_{y},$ then
$\sigma_{z}\sigma_{w}=id$ on the great circle $C$ defined by the
intersection of their equators, this implies
$\sigma_{z}\sigma_{w}$ fixes every point of $\mathbb{R}\times C,$
the contradiction shows that there are exactly two fixed points
$A,B$ of $\Gamma$ lying antipodally on $\mathbb{S}^3.$

Let $G=\Gamma/\Gamma_0$. We claim $G=1.$ Indeed, if $G\neq 1,$ we
pick $1\neq g\in G,$ then $g$ is a fixed point free isometry of
$(\mathbb{R}\times \mathbb{S}^3)/\Gamma_0.$ We must have $g(A)=B$
and $g(B)=A$ and then $g^{2}=1.$  Since $g$ sends geodesics
connecting $A$ to $B$ to geodesics connecting $B$ to $A,$ this
implies $g$ sends the $\mathbb{R}\mathbb{P}^2$ (as quotient of
equator by $\sigma_{z}$ in the above) to itself without fix
points. This is impossible. So we have showed $G$ is trivial. That
means $( \mathbb{S}^3\times \mathbb{R})/\Gamma_0= X.$  By using
our notation in section \ref{2.1},  $X$ is diffeomorphic to
$\mathbb{S}^4/(x,\pm x')\backslash \bar{\mathbb{B}}^4.$ We note
that in this case $X$ has only one end, which is diffeomorphic to
$\mathbb{S}^3\times \mathbb{R}.$
\end{proof}


\subsection{Positive curvature operator case} In this section, we
investigate the canonical neighborhood structure for all cases of
the ancient $\kappa-$orbifold solution.
 If the orbifold admits no singularity,
this has been done in Theorem 3.8 in \cite{CZ05F}. We recall

\begin{theorem}(Theorem 3.8 in \cite{CZ05F})
 For every $\epsilon>0$ one can find positive constants
$C_1=C_1(\epsilon)$, $C_2=C_2(\epsilon)$ such that for each point
$(x,t)$ in every four-dimensional ancient
$\kappa$-\textbf{manifold} solution  (for some $\kappa>0$) with
restricted isotropic curvature pinching and with \textbf{positive}
curvature operator, there is a radius $r$,
$0<r<C_1(R(x,t))^{-\frac{1}{2}}$, so that some open neighborhood
$B_t(x,r)\subset B\subset B_t(x,2r)$ falls into one of the
following three categories:

(a) $B$ is an \textbf{evolving $\epsilon$-neck} (in the sense that
it is the time slice at time $t$ of the parabolic region
$\{(x',t')|x'\in B, t'\in [t-\epsilon^{-2} R(x,t)^{-1},t]\}$ which
 is, after
scaling with factor $R(x,t)$ and shifting the time $t$ to $0$,
$\epsilon$-close (in $C^{[\epsilon^{-1}]}$ topology) to the
 subset $(\mathbb{I}\times \mathbb{S}^{3})\times [-\epsilon^{-2},0]$
 of the evolving round cylinder
$\mathbb{R}\times \mathbb{S}^3$, having scalar curvature one and
length $2\epsilon^{-1}$ to $\mathbb{I}$ at time zero, or

(b) $B$ is an $\textbf{evolving $\epsilon$-cap}$ (in the sense
that it is the time slice at the time $t$ of an evolving metric on
open $\mathbb{B}^4$ or $\mathbb{RP}^4\setminus
\overline{\mathbb{B}^4}$ such that the region outside some
suitable compact subset of $\mathbb{B}^4$ or
$\mathbb{RP}^4\setminus \overline{\mathbb{B}^4}$ is an evolving
$\epsilon$-neck), or

(c) $B$ is a compact manifold (without boundary) with positive
curvature operator
(thus it is diffeomorphic to $\mathbb{S}^4$ or $\mathbb{RP}^4$); \\
furthermore, the scalar curvature of the ancient $\kappa$-solution
in $B$ at time $t$ is between $C^{-1}_2R(x,t)$ and $C_2R(x,t).$

\end{theorem}

The key difficulty in analyzing  the local structure of ancient
$\kappa-$solution is the collapsing of the solution in the presence
of orbifold singularities with big local uniformization groups.
First of all,  we need to generalize the concept of
$\varepsilon-$neck or $\varepsilon-$cap to orbifold solutions with
at most isolated singularities, the point is that we allow a
suitable isometric group to act on the usual necks and caps. Recall
in this paper,
 we  define (topologically) that neck is (diffeomorphic to)  $\mathbb{S}^3/ \Gamma\times \mathbb{
R};$ and smooth cap is  $C_{\Gamma}^{\sigma}$ and  orbifold cap
contains two types:  type I:  $C_{\Gamma},$ and type II:
$\mathbb{S}^4/(x,\pm x')\setminus \bar{B}.$ The motivation to
define the orbifold caps to contain only the above two types is
from the consideration of canonical neighborhoods in this paper.

\begin{definition}

 Fix  $\varepsilon>0$ and a space time point $(x,t).$ Let
   $ B\subset X$ be a space open subset containing $x,$

 (i) we call $B$ an \textbf{evolving
 $\varepsilon$-neck} around $(x,t)$
 if it is the time slice at time $t$ of the parabolic region
$\{(x',t')|x'\in B, t'\in [t-\varepsilon^{-2} R(x,t)^{-1},t]\}$
which satisfies that there is a diffeomorphism $\varphi:
\mathbb{I}\times (\mathbb{S}^{3}/\Gamma)\rightarrow B$ such that
 , after pulling back the solution $(\varphi)^{*}g(\cdot,\cdot)$
 to $\mathbb{I}\times \mathbb{S}^{3},$
scaling with factor $R(x,t)$ and shifting the time $t$ to $0$, the
solution is  $\varepsilon$-close(in $C^{[\varepsilon^{-1}]}$
topology) to the
 subset $(\mathbb{I}\times \mathbb{S}^{3})\times [-\varepsilon^{-2},0]$
 of the evolving round cylinder
$\mathbb{R}\times \mathbb{S}^3$, having scalar curvature one and
length $2\varepsilon^{-1}$ to $\mathbb{I}$ at time zero,

(ii) we call $B$ an $\textbf{evolving $\varepsilon $-cap}$ if it
is the time slice at the time $t$ of an evolving metric on open
smooth caps $C_{\Gamma}^{\sigma}$ and orbifold caps of the above
two types $C_{\Gamma}$ and $\mathbb{S}^4/(x,\pm x')\setminus
\bar{B}$ such that the region outside some suitable compact subset
is an evolving $\varepsilon$-neck around some point in the sense
of (i).
\end{definition}

 Let us start with the following elliptic
type curvature estimate for our orbifold solution. The idea of proof
is to find out a global uniformaization space which is not collapsed
and investigate the isometric group action on it.
\begin{proposition}\label{elliptic}
There is a universal positive function $\omega:[0,
\infty)\rightarrow [0,\infty)$ such that for any an ancient
$\kappa-$orbifold solution on $4$-orbifold $X$, we have
$$
R(x,t)\leq R(y,t)\omega(R(y,t)d_t(x,y)^2)
$$
for any $x,y\in X, t\in (-\infty,0].$
 \end{proposition}

\begin{proof}
This proposition for the case that $X$ is a smooth manifold has been
estalished in \cite{CZ05F} (see Theorem 3.5 and Proposition 3.3 in
\cite{CZ05F}). Thus we may always assume that $X$ has at least one
(orbifold) singularity.

 Case 1: Curvature operator has zero
(eigenvalue) somewhere. Then by section \ref{sec3.2.1}, the scalar
curvature is constant. So the proposition holds trivially in this
case.

Case 2: $X$ is compact with positive curvature operator. By the work
of Hamilton, if we continue to evolve the metric, the metric will
become rounder and rounder. On the other hand, by our
$\kappa-$noncollasing assumption, and the compactness theorem of
\cite{L}, we can extract a convergent subsequence to get a limit
which is compact and round. From this,  we know the orbifold $X$ is
diffeomorphic to a compact orbifold with positive constant sectional
curvature and with at most isolated singularities. By \cite{Th},
there is a finite subgroup $G\subset ISO(\mathbb{S}^4)$ of
isometries of $\mathbb{S}^4$ such that $\mathbb{S}^4/G$ is
diffeomorphic to $X.$ Let $\pi:\mathbb{S}^4\rightarrow X $ be the
naturally defined smooth map, and
$\tilde{g}(\cdot,t)=\pi^{*}g(\cdot,t)$ be the induced $G$ invariant
solution of Ricci flow on smooth manifold $\mathbb{S}^4.$ Now we
check the $\kappa-$noncollapsing of $\tilde{g}.$ Suppose
$\tilde{R}(\cdot,t)\leq r^{-2}$ on $\tilde{B}_{t_0}(\tilde{x},r)$
for all $t\in [t_0-r^2,t_0].$ Let $x=\pi(\tilde{x})\in X,$ $\gamma$
be a geodesic in $X$ of length $\leq r$ with $x=\gamma(0).$ Then
$\gamma$ has a lift
 of geodesic $\tilde{\gamma}$ (which may not be unique) in $\mathbb{S}^4$ with
 $\tilde{\lambda}(0)=\tilde{x}$, and
 $L(\tilde{\gamma})=L(\gamma).$
  This fact implies $\pi:\tilde{B}_{t_0}(\tilde{x},r)\rightarrow
 {B}_{t_0}({x},r)$ is surjective. This implies the curvature
of $X$ is still bounded by $r^{-2}$ on $B_{t_0}(x,r)\times
[t_0-r^{2},t_0],$ and hence $
 vol_{t_0}(\tilde{B}(\tilde{x},r))\geq vol_{t_0}({B}({x},r))\geq \kappa r^{4}$
by the $\kappa-$noncollapsing assumption.  So we have showed that
the solution $\tilde{g}$ is an ancient $\kappa-$solution on smooth
manifold. By \cite{CZ05F} (Theorem 3.5 and Proposition 3.3 in
\cite{CZ05F}), $\tilde{g}(\cdot,t)$ is  $k_0-$noncollapsed for some
universal constant $k_0.$ Furthermore, there is a universal positive
function $\omega$ such that
\begin{equation}\label{{R}}
\tilde{R}(\tilde{x},t)\leq
\tilde{R}(\tilde{y},t)\omega(\tilde{R}(\tilde{y},t)\tilde{d}_{t}(\tilde{x},\tilde{y})^{2})
\end{equation}
for the curvature of induced Ricci flow $\tilde{g}(\cdot,t)$  at
any two points $\tilde{x},\tilde{y}\in \mathbb{R}^4, t\in
(0,\infty].$ For any pair of points $x, y\in X,$  we draw minimal
geodesic $\gamma$ connecting $x,y$ in $X,$  $\gamma$ can be lifted
to a geodesic $\tilde{\gamma}\subset \mathbb{R}^4$ connecting two
points $\tilde{x}\in \tilde{\Phi}^{-1}(x),
\tilde{y}=\tilde{\Phi}^{-1}(y).$ Since
$\tilde{d}(\tilde{x},\tilde{y})\leq L(\tilde{\gamma})=d(x,y)$ and
$R(x,t)=R(\tilde{x},t)$ and $R(y,t)=\tilde{R}(\tilde{y},t),$ by
(\ref{{R}}), we get
$$
R(x,t)\leq {R}({y},t)\omega({R}({y},t){d}_{t}({x},{y})^{2}).
$$

 Case 3: We assume $X$ is noncompact and has positive
curvature operator.  Let $P$ be a fixed singularity of $X.$ We
define a Busemann function $\varphi$ at time $-1$ in the following
way:
$$\varphi(x)=\sup_{\gamma}\lim\limits_{s\rightarrow +\infty}(s-d_{-1}(x,\gamma(s)))
$$
where the sup is taken over all normal geodesic ray $\gamma$
originating from $P.$ It is well-known that $\varphi$ is convex
(with respect to the metric at time $-1$) and of Lipschitz constant
$\leq 1$ and proper. Deforming $\varphi$ by the heat equation
$$
\frac{\partial u}{\partial t}=\triangle_t u
$$
with $u|_{t=-1}=\varphi.$ By a straightforward computation, we have
$$
\frac{\partial}{\partial t} u_{ij}=\triangle
u_{ij}+g^{km}g^{ln}R_{ikjl}u_{mn}-\frac{1}{2}(g^{kl}R_{ik}u_{lj}+g^{kl}R_{jk}u_{lj})
$$
where $u_{ij}=\nabla^{2}_{ij}u$ are the Hessian of $u.$ Noting the
curvature operator is positive, by maximum principle, we have $
\nabla^2 u \geq 0 $ is preserved. Moreover we have $ \nabla^2 u >0 $
at $t=0$ by the following reasons. The kernel of  $ \nabla^2 u$ is a
parallel distribution by strong maximum principle of Hamilton
\cite{Ha86}. If the kernel is nontrivial, then either the space
splits product $\mathbb{R}\times \Sigma$ locally or the space admits
a linear function $(\nabla^2u=0)$. Both cases have contradiction
with the strict positive curvature operator.

 Now we fix the time $t=0.$ Notice that $u$ is still a proper function, so by strict convexity of $u,$ we
know $u$ has a unique critical point, which is the minimal point.
 We claim the minimal point is just the singular point $P$ we
specified in the beginning. Hence there are no other
singularities. The argument is in the following. Let $\pi:
\tilde{U}\rightarrow U,$ $U=\tilde{U}/\Gamma$ be the local
uniformization near $P.$ Then $\tilde{u}=u\circ \pi$ is $\Gamma$
invariant, and we have $d\gamma (\nabla u)(P)=\nabla u(P)$ for any
$\gamma\in \Gamma.$ Since $\Gamma$ has isolated fixed point, we
have $\sum_{\gamma\in \Gamma}d\gamma (\nabla u)(P)=0$ and $\nabla
u(P)=0$ consequently.

Let $\xi=\frac{\nabla u}{|\nabla u|}$ be a vector field which is
singular at $P.$ Now we consider the map
$\Phi:C_pX=Cone(\mathbb{S}^{3}/\Gamma)\rightarrow X$ defined by
$$
\Phi(v,s)=\alpha_{v}(s)
$$
where $\alpha_{v}(s)$ is the integral curve of $\xi$ with
$\alpha_{v}(0)=P$ and $\alpha_{v}'(0)=v.$  By
using$\nabla^{2}u(P)>0,$ the $\frac{\nabla u}{|\nabla u|}$ can
take any value, so the above map is defined.  Clearly, $\Phi$ is a
global orbifold diffeomorphim. We define $\tilde{\Phi}:
\mathbb{R}^4=Cone(\mathbb{S}^3) \rightarrow X$ by
$$
\tilde{\Phi}=\Phi\cdot \pi
$$
where $\pi: Cone(\mathbb{S}^3)\rightarrow
Cone(\mathbb{S}^3/\Gamma)$ is the natural projection. Define
$$\tilde{g}(\cdot,t)=\tilde{\Phi}^{\ast}g(\cdot,t).$$
Then $\tilde{g}(\cdot,t)$ is a smooth complete ancient $\kappa$
solution on smooth manifolds $\mathbb{R}^4$ with positive curvature
operator and restricted isotropic pinching condition. Moreover by
\cite{CZ05F} again,  $\tilde{g}(\cdot,t)$ is  $k_0$ noncollapsed for
some universal constant $k_0,$ and same argument as in Case 2
completes the proof.
\end{proof}

\begin{corollary}\label{3.8}
Let $g_t$ be an ancient $\kappa-$orbifold solution on complete
\textbf{noncompact} $4$-orbifold $X$ with \textbf{positive}
curvature operator and nonempty isolated singularities, then there
is at most one singularity and there is a finite group of isometries
$\Gamma\subset ISO(\mathbb{R}^4)$ of standard $\mathbb{R}^4,$ such
that $O$ is the only fixed point for any element of $\Gamma,$ and
$X$ is diffeomorphic to $\mathbb{R}^4/\Gamma$ as orbifolds.

\end{corollary}

\begin{theorem}\label{t3.9}
 For every $\varepsilon>0$ one can find positive constants
$C_1=C_1(\varepsilon)$, $C_2=C_2(\varepsilon)$ such that for each
point $(x,t)$ in every complete \textbf{noncompact}
four-dimensional ancient $\kappa$-\textbf{orbifold} solution  with
\textbf{positive} curvature operator, there is a radius $r$,
$\frac{1}{C_1}(R(x,t))^{-\frac{1}{2}}
<r<C_1(R(x,t))^{-\frac{1}{2}}$, so that some open neighborhood
$B_t(x,r)\subset B\subset B_t(x,2r)$ falls into one of the
following two categories:

(a) $B$ is an \textbf{evolving $\varepsilon$-neck} around $(x,t)$,

(b) $B$ is an $\textbf{evolving $\varepsilon $-cap}$ of Type I.

Moreover, the scalar curvature  in $B$ at time $t$ is between
$C^{-1}_2R(x,t)$ and $C_2R(x,t).$

\end{theorem}
\begin{proof}
 We denote the unique singularity by $O.$ By corollary \ref{3.8}, $X$ is diffeomorphic to
$\tilde{X}/\Gamma,$ where $\tilde{X}$ is diffeomorphic to
$\mathbb{R}^4,$ and $\Gamma\subset ISO(\mathbb{R}^4)$ fixes the
origin, denoted also by $O.$ Let $\tilde{g}$ be the pulled back
solution on $\tilde{X},$ which is a $\Gamma-$ invariant solution on
$\tilde{X}.$ Note that the solution $\tilde{g}$ is also
$\kappa-$noncollapsed and therefore $\kappa_0-$noncollapsed for some
universal $\kappa_0>0$ by theorem 3.5 in \cite{CZ05F}. Fix time
$t=0.$ Now by the proof of Theorem 3.8 in \cite{CZ05F}, there is a
point $x_0\in \tilde{X},$ such that for any given small $\epsilon>0,
$ there is a constant $D(\epsilon)>0$ depending only on $\epsilon$
such that any $(x,0)$ satisfying $R(x_0,0)d_{0}(x,x_0)^{2}\geq
D(\epsilon)$ admits an {evolving $\epsilon$-neck} around it. We
scale the solution so that $R(x_0,0)=1.$ In the following, we
describe the canonical parametrization of necks which was given by
 Hamilton in the section C of \cite{Ha97}.
 We will use Hamilton's canonical parametrization to parametrize all the points outside a ball of radius $D(\epsilon)+1$ centered at $x_0$ by a canonical
 diffeomorphism $\Phi$ from $\mathbb{S}^3\times \mathbb{I},$ where $\mathbb{I}\in \mathbb{R}$ is an interval.

   For any $z\in \tilde{X}$ with $d_{0}(z,x_0)^{2}\geq D(\epsilon),$
 there is a unique constant mean
curvature hypersurface $S_z\in \tilde{X} $ passing through $z.$
Each $(S_z,\tilde{g})$ can be parametrized by a harmonic
diffeomorphism from standard sphere $(\mathbb{S}^3,\bar{g})$ to
it, since the (induced) metrics $\tilde{g}$ and $\bar{g}$ is very
close. The coordinate function of factor $\mathbb{R}\ni s$ can
also be uniquely chosen in the following way: let
$$ area(\mathbb{S}^{3}\times \{s\},\tilde{g})=vol(\mathbb{S}^{3})r(s)^{3},$$
we require function  $s$  satisfies
$$Vol(\mathbb{S}^{3}\times
[s_1,s_2],\tilde{g})=vol(\mathbb{S}^{3})\int_{s_1}^{s_2}r(s)^{3}ds.$$
Notice that the above harmonic diffeomorphisms are unique up to a
rotations of $(\mathbb{S}^3,\bar{g}),$ since the induced metrics
are close to the standard one. We require if $\bar{V}$ is an
infinitesimal rotation on $(\mathbb{S}^3\times \{z\},\bar{g}),$
and $W$ is the unit vector field which is $\tilde{g}$ orthonormal
to the sphere $\mathbb{S}^3\times \{z\},$ then
\begin{equation}\label{e2}\int_{S^3\times
\{z\}}\bar{g}(\bar{V},W)=0.\end{equation}

The above  parameterization $\Phi:\mathbb{S}^3\times
(A,B)\rightarrow \tilde{X}$ can be extended on one end so that it
covers  all points outside a ball of radius $D(\epsilon)+1$
centered at $x_0.$ Without loss of generality, we assume as
$z\rightarrow B,$ the points on the manifold $\tilde{X}$ divergent
to infinity.

 Let $\hat{g}=\Phi^{\ast}\tilde{g}.$ Let $\gamma\in \Gamma,$ and $\hat{\gamma}=\Phi^{-1}\gamma\Phi.$
 Since $\gamma$ is an isometry of
$\tilde{g},$ it sends constant mean curvature spheres to constant
mean curvature spheres. So $\hat{\gamma}$ preserves the foliation
of the horizontal spheres. So the uniqueness of harmonic maps in
this case implies $\hat{\gamma}$ is isometry in $\mathbb{S}^3$
factor.
 The specific choice of coordinate $s\in
\mathbb{R}$ implies the $\mathbb{R}$ component of $\hat{\gamma}$
is an isometry of $\mathbb{R},$ and independent of the factor
$\mathbb{S}^3.$ The (\ref{e2}) straighten out the rotations so
that they align themselves to a global isometry of the standard
$\mathbb{S}^3\times \mathbb{I},$ $\mathbb{I}=(A,B).$ So the group
$\hat{\Gamma}=\Phi^{-1}\Gamma \Phi$ acts isometrically on
$\mathbb{S}^3\times \mathbb{I}$ with the standard metric. We claim
 the $\mathbb{R}$ factors of $\hat{\Gamma}$ have no translations.
 Indeed, suppose there is one $\hat{\gamma}\in \hat{\Gamma}$ such that the $\mathbb{R}$ factor of $\hat{\gamma}$ is a
 translation $s\mapsto s+L$ with $L>0.$ Otherwise we consider $\hat{\gamma}^{-1}$.
 So any point in finite region will be mapped to very far by
 $\hat{\gamma}^{m}$ as $m\rightarrow \infty.$ Since the
 $\hat{\gamma}^{m}$ are isometries, and the manifold at infinity
 splits off a line, we conclude the curvature operator is not
 strictly positive in finite region. This is a contradiction with
 our assumption. The $\mathbb{R}$ factors of $\hat{\Gamma}$ also contain  no reflections,
 otherwise the manifold will contain two ends and splits off a line
 globally. So we conclude that $\hat{\Gamma}$ only acts on the
 factor $\mathbb{S}^3.$ This implies that the above
 parametrization descents to a parametrization $\phi: \mathbb{S}^3/\Gamma\times
 (A,B)\rightarrow X.$

 Since as $\varepsilon\rightarrow0,$ after normalizations,
  metric $\hat{g}$ will converge in $C^{\infty}_{loc}$ topology to the standard one. This implies the following fact: for any given $\varepsilon>0,$
 there is $\tilde{\varepsilon}>0$ such that if  $\epsilon<\tilde{\varepsilon},$  then for any point $P\in \mathbb{S}^3\times
 (A,B),$ the metric $\hat{g}$ on $\mathbb{S}^3\times
 (A,B)$  around $P$ is $\varepsilon-$ close to the standard one after
 scaling with the  factor $\hat{R}(P).$

We can also show the point $O$ has distance $\leq
\sqrt{D(\epsilon)}+1$ with $x_0.$ Indeed, if $d_0(x_0,O)\geq
\sqrt{D(\epsilon)}+1,$ then $O$ is covered by the parameterization
$\Phi:\mathbb{S}^3\times (A,B)\rightarrow \tilde{X}.$ Let
$O=\Phi(\bar{x},\bar{s}),$ $\bar{x}\in \mathbb{S}^3,$ $\bar{s}\in
(A,B).$ Since the group $\hat{\Gamma}$ only acts on the
 factor $\mathbb{S}^3,$ we conclude that $\hat{\Gamma}$ fixes every point on $\{\bar{x}\}\times (A,B).$ This is a contradiction.

Now we are ready to prove the theorem. For the given
$\varepsilon>0,$ there is a $\tilde{\varepsilon}>0$ defined in the
above. For any point $x\in X$ with $d_0(O,x)\geq
2D(\frac{1}{2}\tilde{\varepsilon}),$ a suitable portion
$\mathbb{S}^3/\Gamma\times (A'B')$ of $\mathbb{S}^3/\Gamma\times
(A,B)$ in the above parametrization will give a $\varepsilon-$ neck
neighborhood of $x.$ Let $\tilde{x}\in \tilde{X}$ satisfy
$d_0(\tilde{x},O)=10D(\frac{1}{2}\tilde{\varepsilon}),$ denote the
constant mean curvature hypersurface passing through $\tilde{x}$ by
$\Sigma.$ By Theorem G1.1 in \cite{Ha97}, $\Sigma$ bounds a open set
$\Omega$ which is  differentiable ball $\mathbb{B}^4$ in
$\tilde{X}.$ $\Omega$ is $\Gamma-$ invariant, and  $\Omega/\Gamma$
contains an $\varepsilon-$ neck with its end. The curvature estimate
on $\Omega/\Gamma$ follows from the above Proposition
\ref{elliptic}. Thus we only need to show $\Omega/\Gamma$ is
diffeomorphic to the orbifold cap $C_{\Gamma}$ of type I.

Let $\varphi: X\rightarrow \mathbb{R}$ be the Busemann function at
time $t=0$ on $X$ constructed around the singular point $O.$ Let
$u_{\delta}$ be a family of strictly convex smooth perturbation of
$\varphi$ as in Proposition \ref{elliptic} such that
$u_{0}=\varphi.$ By considering the integral curves of
$u_{\delta}$ as in Proposition \ref{elliptic}, one can show  the
level sets  $u_{\delta}^{-1}(c)$ of $u_{\delta}$ diffeomorphic to
$C_{\Gamma}.$

  Let
$f$ be the function of coordinate $\mathbb{R} $ on the
parametrization $\phi:\mathbb{S}^3/\Gamma\times (A,B)\rightarrow
X.$ By a geometric argument, one can show $\nabla \varphi$ is
almost parallel (with error controlled by $\varepsilon$) to the
$\nabla f,$ and so does $\nabla u_{\delta}$ for small $\delta.$ By
blending the function $u_{\delta}$ and a multiple of $f$ by a bump
function, we get a function $\psi,$ whose gradient curves gave a
diffeomorphism from $f^{-1}(-\infty,c']$ and
$u_{\delta}^{-1}(-\infty,c]$ by Morse theory. This particularly
shows that $\Omega/\Gamma$ is diffeomorphic to $C_{\Gamma}.$
 The
proof of the theorem is completed.
\end{proof}

We summarize the results obtained in this section:
\begin{theorem} \label{t3.10}
 For every $\varepsilon>0$ one can find positive constants
$C_1=C_1(\varepsilon)$, $C_2=C_2(\varepsilon),$ such that for
every four-dimensional ancient $\kappa$-\textbf{orbifold} solution
$(X,g_t),$   for each point $(x,t),$ there is a radius $r$,
$\frac{1}{C_1}(R(x,t))^{-\frac{1}{2}}
<r<C_1(R(x,t))^{-\frac{1}{2}}$, so that some open neighborhood
$B_t(x,r)\subset B\subset B_t(x,2r)$ falls into one of the
following two categories:

(a) $B$ is an \textbf{evolving $\varepsilon$-neck} around $(x,t),$

(b) $B$ is an $\textbf{evolving $\varepsilon $-cap},$

(c) $X$ is diffeomorphic to a closed spherical orbifold
$\mathbb{S}^4/\Gamma$ with at most isolated singularities.

Moreover, the scalar curvature in $B$ at case (a) and (b) at time
$t$ is between $C^{-1}_2R(x,t)$ and $C_2R(x,t).$

\end{theorem}
\begin{proof}
By Theorem \ref{t3.4} and Theorem \ref{t3.9}, we only need to
consider the case when $X$ is compact with positive curvature
operator. In this case, we continue to evolve the metric by Ricci
flow. Since the scalar curvature is strictly positive, the solution
will blow up in finite time.  By using the $\kappa-$noncollapsing in
\cite{P1} and the compactness theorem in \cite{L}, we can scale the
solution in space time around a sequence of points and extract a
convergent subsequence. Moreover, the limit is still an orbifold
with at most isolated singularities by \cite{L}. By the pinching
estimate of Hamilton \cite{Ha86}, the Riemannian metric in the limit
orbifold has constant sectional curvature.  So it is a global
quotient of sphere.
\end{proof}

\section{Surgerical solutions}
\label{4}
\subsection{Surgery at first singular time}\label{4.1}
Since the scalar curvature at initial time is strictly positive, it
follows from the maximum principle and the evolution equation of the
scalar curvature that the curvature must blow up at some finite time
$0<T<\infty$. Note that the canonical structures of ancient
$\kappa-$orbifold solutions have been completely described in the
last section. Combining with a technical geometric lemma
(Proposition \ref{p5.1} in the appendix), we have the similar
singularity structure theorem before time $T$ as in the manifold
case (see Theorem 4.1 in \cite{CZ05F}).

\begin{theorem} \label{t4.1}
Given small $\varepsilon>0$, there is  $r=r(T)>0$ depending on
$\varepsilon, T$ and the initial metric such that for any point
$(x_0,t_0)$ with $Q=R(x_0,t_0)\geq r^{-2}$, the solution in the
parabolic region $\{(x,t)\in X\times [0,T)|
d^2_{t_0}(x,x_0)<\varepsilon^{-2} Q^{-1},t_0-\varepsilon^{-2}
Q^{-1}<t\leq t_0\}$ is, after scaling by the factor $Q$,
$\varepsilon$-close (in $C^{[\varepsilon^{-1}]}$-topology) to the
corresponding subset of some ancient $\kappa$-orbifold solution with
restricted isotropic curvature pinching (\ref{ricp}) and with at
most isolated orbifold singularities.
\end{theorem}
\begin{proof} First of all, we may
assume the orbifold is not diffeomorphic to a spherical orbifold
$\mathbb{S}^4/\Gamma,$ otherwise we are in case c) in Theorem
\ref{t3.10}. We argue by contradiction as in manifold case
\cite{CZ05F}.

 We choose a
 point $(x_0,t_0)$ almost critically violating the conclusion of the
theorem.  We scale the solution around $(x_0,t_0)$ with factor
$R(x_0,t_0)$ and shift the time $t_0$ to $0.$ The key point of the
proof is to bound the curvature.   Note that we still have
$\kappa-$noncollapsing condition (Lemma\ref{ka}), and compactness
theorem \cite{L} for $\kappa-$noncollapsed Ricci flow solutions on
orbifolds with isolated singularities.  By the canonical
neighborhood decomposition theorem for ancient $\kappa-$ solutions,
we can show the curvature is bounded in bounded normalized  distance
with $x_0.$ The boundedness of curvature on the limit space follows
from Proposition \ref{p5.1}. We have all the ingredients we need to
mimic the same proof in the manifold case \cite{CZ05F} to show that
we can extract a convergent subsequence which converges to an
ancient $\kappa$-orbifold solution. This is a contradiction.
\end{proof}

We denote by $\Omega$ the open set of points where curvature become
bounded as $t\rightarrow \infty.$ Denote by $\bar{g}$ the limit of
$g_t$ on $\Omega$ as $t\rightarrow T.$

 Fix $0<\delta<<\varepsilon,$ and let $\rho=\rho(T)=\delta r(T),$
    and $\Omega_{\rho}=\{x\in X\mid \bar{R}\leq \rho^{-2}\}$.
If $\Omega_{\rho}$ is empty, then by Theorem \ref{t4.1} and
Theorem \ref{t3.10}, $X$ is either diffeomorphic to a spherical
orbifold $\mathbb{S}^{4}/\Gamma$ with at most isolated
singularities, or $X$ is covered by $\varepsilon-$ necks and
$\varepsilon-$ caps. For the latter case, if there occurs no caps,
$X$ is covered by $\varepsilon-$ necks, hence diffeomorphic to
$\mathbb{S}^3/\Gamma\times \mathbb{S}^1$ or
$\mathbb{S}^3/\Gamma{\times}_{f}\mathbb{S}^1$; if there are caps,
we have four types of caps: $C_{\Gamma}^{\sigma},$ $C_{\Gamma},$
$\mathbb{S}^{4}/(x,\pm x')\setminus \bar{\mathbb{B}}^4,$
$\mathbb{B}^4$ and hence $X$ is diffeomorphic to either smooth
manifolds $\mathbb{S}^4,$ $\mathbb{R}\mathbb{P}^4,$
$C_{\Gamma}^{\sigma}\cup_{f}
  C_{\Gamma'}^{\sigma'},$or one of the orbifolds
  $C_{\Gamma}^{\sigma}\cup_{f}
  C_{\Gamma'},$ $C_{\Gamma}\cup_{f} C_{\Gamma'},$ $\mathbb{S}^4/(x,\pm
  x'),$ $\mathbb{S}^4/(x,\pm x')\#\mathbb{R}\mathbb{P}^4,$ $\mathbb{S}^4/(x,\pm x')\# \mathbb{S}^4/(x,\pm x').$
So we conclude that if $\Omega_{\rho}$ is empty, the $X$ is
diffeomorphic to a spherical orbifold $\mathbb{S}^{4}/\Gamma$ with
at most isolated singularities or a connected sum of
  two spherical
   orbifolds $\mathbb{S}^4/\Gamma_1$ and $\mathbb{S}^4/\Gamma_2$ with at most isolated
   singularities. While if the solution, near the time $T$, has
   positive curvature operator, it follows from the proof of Theorem
   \ref{t3.10} that $X$ is diffeomorphic to a spherical orbifold with
at most isolated singularities. Thus, when $\Omega_{\rho}$ is empty
or the solution becomes to have positive curvature operator
everywhere, we stop the the procedure here and say that the solution
becomes extinct.

We then may assume that $\Omega_{\rho}\neq \phi$ and any point
outside $\Omega_{\rho}$ has a $\varepsilon-$neck or
$\varepsilon-$cap neighborhood. We are interested in those
$\varepsilon-$horns $H$ ( consisting of $\varepsilon-$necks) whose
one end is in $\Omega_{\rho}$ and the curvature becomes unbounded on
other end. We will perform surgeries on these horns. First of all,
we need the existence of finer necks(than $\varepsilon$) in the
$\varepsilon-$horn $H.$ The reason to find a finer neck to perform
surgeries is to quantitatively control the accumulations of the
errors caused by surgeries.

\begin{proposition}\label{p4.2}For the arbitrarily given small $0<\delta<<\varepsilon,$
there is an $0<h< \delta\rho$ depending only on $\delta$ and
$\varepsilon,$ and independent of non-collapsing parameter $\kappa$
such that if a point $x$ on the $\varepsilon-$horn $H$ whose finite
end is in $\Omega_{\rho}$ has curvature $\geq h^{-2},$ then there is
a $\delta-$neck around it.
\end{proposition}
 The argument is a bit different from Lemma 5.2 in \cite{CZ05F}. The reason is that, the canonical
 neighborhoods in \cite{CZ05F}
 are universally non-collapsed, but in the present situation we do not know it a priori.
 \begin{proof}There is a fixed point free
finite group of isometries $\Gamma\in ISO(\mathbb{S}^3)$ so that we
can apply Hamilton's parametrization to parametrize the whole $H$,
$\Phi_{\Gamma}: (\mathbb{S}^3/\Gamma)\times (A,B) \rightarrow H,$
where $\Phi_{\Gamma}$ is a diffeomorphism. Denote by $\Phi:
\mathbb{S}^3\times (A,B)\rightarrow H$ the natural projection.
Without loss of generality, we assume $\Phi(
\mathbb{S}^3\times\{s\})$ has nonempty intersection with
$\Omega_{\rho}$ as $s\rightarrow A,$ and curvature becomes unbounded
as $s\rightarrow B.$ To prove the claim, we argue by contradiction.
Suppose ${x_j}\in H$ is a sequence of points with
$\bar{R}({x_j})\geq h^{-2}\rightarrow \infty$ but $x_j$ has no
$\delta-$neck neighborhood. We pull back the solution to
$\mathbb{S}^3\times (A,B),$ scale with factor $\bar{R}(x_j)$ around
$x_j,$ shift the time $T$ to $0.$  Note the rescaled solution on
$\mathbb{S}^3\times(A,B) $ is smooth (without orbifold
singularities) and uniformly non-collapsed. We apply the same
argument of step 2 in Theorem 4.1 in \cite{CZ05F} to show that the
curvature is bounded in any fixed finite ball around point $x_j$ for
the rescaled solution, otherwise we get a piece of non-flat
nonnegatively curved metric cone as a blow up limit, which
contradicts with Hamilton's strong maximum principle(see
\cite{Ha86}). This implies the two ends of $\mathbb{S}^3\times
(A,B)$ are very far from point $x_j$ (in the normalized distance).
We then extract (around $(\bar{x}_j,T)$) a convergent subsequence so
that the limit splits off a line  by the Toponogov splitting
theorem. By (\ref{restrictive}), the limit is the standard
$\mathbb{S}^3\times \mathbb{R}.$ Since the solution is $\Gamma$
invariant, it descends to $H$ and gives a $\delta-$neck around $x_j$
as $j$ large enough. This is a contradiction.
\end{proof}

Let us describe the Hamilton's surgery along the $\delta-$neck $N$
with scalar curvature $h$ in the center $\bar{x}.$  We assume the
normalization (of $ \mathbb{S}^3\times (A,B)$ with some factor) is
so that the metric $h^{-2}\Phi^{*}g$ on $(\bar{x},s_0)\in
\mathbb{S}^3\times (A,B)$ is $\delta'$ close to standard neck metric
$ds^2=dz^2+ds^2_{\mathbb{S}^3}$ on $ \mathbb{S}^3\times\mathbb{R}$
of scalar curvature 1, where $\delta'=\delta'(\delta)$ satisfies
$\lim\limits_{\delta\rightarrow0}\delta'=0.$   We assume the center
of the $\delta-$neck has $\mathbb{R}$ coordinate $z=0.$ The surgery
is to cut open the neck (in Hamilton's parametrization) and glue
back  caps $(\mathbb{B}^4,\tilde{g})$ by conformal pinching the
metric $\bar{g}$ and bending it with the standard cap metric (see
\cite{CZ05F} and \cite{Ha97}). We describe the construction on the
left hand (of coordinate $\mathbb{R}$)(corresponding to the finite
part connectting to $\Omega_{\rho}$)
$$
   \tilde{g} = \left\{
   \begin{array}{lll}
    \bar{g}, \ \ \ z= 0,
         \\[4mm]
    e^{-2f}\bar{g}, \ \ \ z \in [0,2], \\[4mm]
    \varphi e^{-2f}\bar{g} + (1-\varphi)e^{-2f}h^2g_0, \ \ \ z \in [2,3], \\[4mm]
    h^2e^{-2f}g_0, \ \ \ z\in [3,c'],
\end{array}
\right.
$$
where $f$ is some fixed function and $g_0$ is the standard metric.
We also perform the same surgery procedure on the right hand with
parameters $\tilde{z}\in [0,4]$ ($\tilde{z}=8-z$).

 Since the group $\Gamma$ acts isometrically on the factor
$\mathbb{S}^3$ of $\mathbb{S}^3\times\mathbb{R} ,$ the above surgery
procedure on  Hamilton's parametrization descends to a surgery on
the space $X$ by cutting off a $\delta-$neck and gluing back two
orbifold caps $C_{\Gamma}$ separately.  We call the above procedure
as a $\textbf{$\delta$-cutoff surgery}$.

 Now
at least the proof of justification of pinching estimates of
Hamilton can be carried through without changing a word.
\begin{lemma}\label{jup} (Hamilton \cite{Ha97} D3.1, Justification
of the pinching assumption)

 There are universal positive constants $\delta_0$,  such that for any $\tilde{T}$ there is a constant $h_0>0$
depending on the initial metric and $\tilde{T}$ such that if we
perform above $\delta$-cutoff surgery at a $\delta$-neck of radius
$h$ at time $T\leq\tilde{T}$ with $\delta < \delta_0$ and $h^{-2}
\geq h_0^{-2}$,  such that after the surgery, the pinching
condition (\ref{pinch}) still holds at all points at time $T$.
\end{lemma}

\subsection{A priori assumptions} We can define the notion of Ricci
flow with surgeries in the same way as in \cite{CZ05F} by replacing
manifolds with orbifolds with at most isolated singularities. As in
\cite{CZ05F}, the solutions to Ricci flow with surgery in this paper
are obtained by performing concrete surgeries.  We cut open  a neck
in a horn and glue back two caps. This makes the all connected
components after surgeries are also closed orbifolds with at most
isolated singularities. Notice that each neck in the horn is
diffeomorphic to $\mathbb{S}^3/\Gamma.$ If $\Gamma$ is trivial, we
glue the usual caps $B^{4}$; if $\Gamma$ is no trivial, we glue back
orbifold caps $C_{\Gamma},$ this produces new orbifold singularities
(tips of the caps).

To understand the topology, we are interested in the solutions
with good properties. Namely, we would like to construct a long
time solution satisfying the a priori assumptions consisting of
the pinching assumption and the canonical neighborhood assumption.

 $\underline{\mbox{\textbf{Pinching assumption}}}$: \emph{There
exist positive constants $\rho,\Lambda,P<+\infty$ such that there
hold
$$a_1+\rho>0 \mbox{ and } c_1+\rho>0, \eqno (5.1)$$ $$\max\{a_3,b_3,c_3\}
\leq \Lambda(a_1+\rho) \mbox{ and }\max\{a_3,b_3,c_3\}\leq
\Lambda(c_1+\rho), \eqno (5.2)$$ and
$$\frac{b_3}{\sqrt{(a_1+\rho)(c_1+\rho)}}\leq 1+\frac{\Lambda
e^{Pt}}{\max\{\log\sqrt{(a_1+\rho)(c_1+\rho)},2\}}, \eqno (5.3)$$
 everywhere.} \vskip 0.3cm

$\underline{\mbox{\textbf{Canonical neighborhood assumption (with
accuracy $\varepsilon$)}}}$:  \emph{ Let $g_t$ be a solution to
the Ricci flow with surgery staring with (\ref{3.1}). For the
given $\varepsilon>0$, there exist two constants
$C_1(\varepsilon)$, $C_2(\varepsilon)$ and a non-increasing
positive function $r$ on $[0,+\infty)$ with the following
properties. For every point $(x,t)$ where the scalar curvature
$R(x,t)$ is at least $r^{-2}(t)$, there is an open neighborhood
$B$, $B_t(x,r)\subset B\subset B_t(x,2r)$ with
$0<\sigma<C_1(\varepsilon)R(x,t)^{-\frac{1}{2}}$,
 which falls into
one of the following three categories:}

\emph{(a) $B$ is a strong $\varepsilon$-neck,
 }

 \emph{(b) $B$ is an $\varepsilon$-cap,  }

\emph{(c) at time $t$, $X$ is diffeomorphic to a closed spherical
orbifold $\mathbb{S}^4/\Gamma$ with at most isolated
singularities.\\ Moreover, for (a) and (b), the scalar curvature
in $B$ at time
 $t$ is between $C^{-1}_2R(x,t)$ and $C_2R(x,t)$, and satisfies
 the gradient estimate $$|\nabla R|<\eta R^{\frac{3}{2}}\mbox{ and }|
 \frac{\partial R}{\partial t}|<\eta R^2,$$} \emph{where $\eta$ is a
 universal constant and the definitions of $\varepsilon-$cap and strong $\varepsilon-$neck will be given in the next paragraph.}

 We give the precise definitions  of  $\varepsilon-$cap, and strong $\varepsilon-$neck in the following. First, we say an open set $B$ on an orbifold
is an \textbf{$\varepsilon-$neck}
 if there is a diffeomorphism $\varphi:
\mathbb{I}\times (\mathbb{S}^{3}/\Gamma)\rightarrow B$ such that
 the pulled back metric $(\varphi)^{*}g,$
 scaling  with some factor, is  $\varepsilon$-close (in
$C^{[\varepsilon^{-1}]}$ topology) to the standard metric
$\mathbb{I}\times (\mathbb{S}^{3}/\Gamma)$ with scalar curvature 1
and $\mathbb{I}=(-\varepsilon^{-1}, \varepsilon^{-1}).$ An open
set $B$ is  \textbf{$\varepsilon-$cap} if $B$ is diffeomorphic to
smooth cap $\mathbb{B}^4,$ $C_{\Gamma}^{\sigma},$
 orbifold cap of Type I , II,  $C_{\Gamma}$ or $\mathbb{S}^{4}/(x,\pm x')\backslash \bar{\mathbb{B}}^4,$
 and the region around the end is an
 $\varepsilon-$neck. A \textbf{strong $\varepsilon-$neck} $B$ at $(x,t)$
 is the time slice at time $t$ of the parabolic region
$\{(x',t')|x'\in B, t'\in [t- R(x,t)^{-1},t]\}$ where the solution
is well-defined and has the property that there is a diffeomorphism
$\varphi: \mathbb{I}\times (\mathbb{S}^{3}/\Gamma)\rightarrow B$
such that
 , the pulling back solution $(\varphi)^{*}g(\cdot,\cdot)$
 scaling with factor $R(x,t)$ and shifting the time $t$ to $0$,  is  $\varepsilon$-close(in $C^{[\varepsilon^{-1}]}$
topology) to the
 subset $(\mathbb{I}\times \mathbb{S}^{3}/\Gamma)\times [-1,0]$
 of the evolving round cylinder
$\mathbb{R}\times (\mathbb{S}^3/\Gamma)$, having scalar curvature
one and length $2\varepsilon^{-1}$ to $\mathbb{I}$ at time zero.

In order to take limits for surgerical orbifold solutions, we need
the noncollapsed condition.  Let $\kappa$ be a positive constant. We
say the solution is \textbf{$\kappa-$noncollapsed on the scales less
than $\rho$} if it satisfies the following property: if $$
|Rm(\cdot,\cdot)| \leq r^{-2}$$  on $ P(x_0,t_0,r,-r^2)=\{ (x',t') \
|\ x'\in B_{t'}(x_0,r), t' \in [t_0 -r^2,t_0] \}$ and $r < \rho$,
then we have
$$Vol_{t_0}(B_{t_0}(x_0,r)) \geq \kappa r^4.$$  Since we are
dealing with solutions with surgeries, the parabolic neighborhood
$P(x_0,t_0,r,-r^2)$ is a little bizarre, the condition $ |Rm(x,t)|
\leq r^{-2}$ is imposed on the place where the solution is
defined.

We will inductively construct a long time solution $g(t)$ satisfying
the a priori assumptions. In section \ref{4.1},  we actually have
constructed a solution satisfying a priori assumptions for a period
of time. In order to  extend our solution for a longer time
inductively, we need to do surgery repeatly. In particular, we need
that there exist sufficient fine necks in horns of surgical
solutions and the estimate has to be quantitative. The following
statement is similar to Proposition \ref{p4.2}, but the situation is
a bit different, since we are dealing with solutions with surgery.

\begin{proposition}
Suppose we have a solution to the Ricci flow with surgery on $(0,T)$
satisfying the a priori assumptions in the above, and the solution
becomes singular as $t\rightarrow T.$  For the arbitrarily given
small $0<\delta<<\varepsilon,$ there is an $0<h< \delta \rho(T)
=\delta^2r(T)$ depending only on $\delta,$ $\varepsilon,$ and $r(T)$
such that if at the time $T$, a point $x$ on a $\varepsilon-$horn
$H$ whose finite end is in $\Omega_{\rho(T)}$ has curvature $\geq
h^{-2},$ then there is a $\delta-$neck around it.
\end{proposition}\label{4.3}
\begin{proof} We observe that the canonical neighborhoods of the points in the
$\varepsilon-$ horn $H$ (far from the end) are all strong
$\varepsilon-$ necks. The the solution around any point $\bar{x}$ on
$H$ with ${R}(\bar{x},T)\geq h^{-2}$ has existed for a previous time
interval $(T-{R}(\bar{x},T)^{-1},T)$.  Suppose the proposition is
not true.   We use Hamilton's parametrization
$\Phi:\mathbb{S}^3\times (A,B)\rightarrow H $ to pull back the
solution on $\mathbb{S}^3\times (A,B).$ By the same argument of
Proposition \ref{p4.2}, we extract a convergent subsequence from the
parabolic scalings around suitable points $\bar{x}$ with
${R}(\bar{x},T)\geq h^{-2}\rightarrow \infty.$ The limit solution is
just the standard solution on $\mathbb{S}^3\times \mathbb{R}$ which
exists at least on the time interval $(-1,0]$ after shifting the
origin. Moreover, the solution on all points (on the original space)
at normalized time $-1+\frac{1}{100}$ still has strong
$\varepsilon-$neck neighborhoods and the scalar curvature is $\leq
1$ as $h^{-1}\rightarrow \infty.$ So we can actually extract a
subsequence so that the limit solution is defined at least on
$[-2,0].$ Since the solution is $\Gamma-$invariant, this gives a
$\delta-$neck as $h^{-1}$ is very large. This is  a contradiction.
\end{proof}
Now we justify the uniform $\kappa-$noncollapsing under the
assumption of canonical neighborhoods with accuracy $\varepsilon$
for some parameter $\tilde{r}$ which may be very small.  The key
point is that even if we perform $\delta-$cutoff surgeries with
sufficient fine $\delta$ which depends on $\tilde{r},$ the
noncollapsing constant $\kappa$ we obtained is uniform and
independent of $\tilde{r}$. In Lemma 5.5 in \cite{CZ05F}, the same
estimate was deduced when the space is smooth. The fact that the
canonical neighborhoods in \cite{CZ05F} are not collapsed played a
crucial role in the proof there. In the current context, at a
priori, the canonical neighborhoods may be sufficiently collapsed.
We need a different argument. Our idea is the following. When the
scale is not too small comparing with the canonical neighborhood
parameter $\tilde{r},$ we observe that the surgery is performed
far away and the argument of Perelman's Jacobian comparison
theorem can be modified to apply as in the smooth case
\cite{CZ05F}. When the scale is small, we first show the space has
a canonical geometric neck near the point and then extend the
canonical geometric neck to form a long geometric tube so that the
other end of the tube has a neck of big scale. After showing the
neck with big scale is noncollapsing, we will get a control on the
order of the fundamental group of the neck which in turn gives the
control on the noncollapsing of the original neck with small
scale.

\begin{lemma} \label{kappan}Given a compact four-orbifold with positive isotropic curvature and given
small $\varepsilon>0$ and a positive integer $l$. Suppose we have
constructed the sequences $\delta_0>0,$ $\tilde{\delta_j}>0,$
$r_j>0,$ $\kappa_j>0,$ $ 0\leq j\leq l-1,$  such that any solution
to the Ricci flow with
 surgery on $[0,T),$ with $T\in
[l\varepsilon^2,(l+1)\varepsilon^2)$ and with the four-orbifold as
the initial data, obtained by $\delta(t)$-cutoff surgeries with
$\delta(t)\leq \delta_0,$ satisfies the following three properties:

(i)  the pinching assumption holds on $[0,T),$

(ii) if $\delta(t)\leq \tilde{\delta}_j$ on
$[j\varepsilon^2,(j+1)\varepsilon^2],$ for all $0\leq j\leq l-1,$
then the canonical neighborhood assumption (with accuracy
$\varepsilon$) holds with parameter $r_j>0$ on each
$[j\varepsilon^2,(j+1)\varepsilon^2]$ for all $0\leq j\leq l-1;$

(iii) if $\delta(t)\leq \tilde{\delta}_j$ on
$[j\varepsilon^2,(j+1)\varepsilon^2],$ for all $0\leq j\leq l-1,$
then it is $\kappa_{j}>0$ noncollapsed  on
$[j\varepsilon^2,(j+1)\varepsilon^2]$ for all scales less than
$\varepsilon,$ for all $0\leq j\leq l-1.$

Then there exists a
$\kappa_{l}=\kappa_{l}(\kappa_{l-1},r_{l-1},\varepsilon)>0$ and for
any $\tilde{r}>0,$ there exists
$\tilde{\delta}_{l}=\tilde{\delta}_{l} (\kappa_{l-1},
\tilde{r},\varepsilon)>0$ such that any solution to Ricci flow with
$\delta(t)-$cutoff surgeries on $[0,T')$ for some $T'\in
[l\varepsilon^2,(l+1)\varepsilon^2]$ is $\kappa_{l}$-noncollapsed on
$[(l-1)\varepsilon^2,T')$ for all scales less than $\varepsilon,$ if

(a) it satisfies the canonical neighborhood assumption (with
accuracy $\varepsilon$) with parameter $\tilde{r}$ on
$[l\varepsilon^2, T');$

(b) for each $t\in [l\varepsilon^2,T'),$ on each connected
components of the solution, there is a point $x$ on it such that
$R(x,t)\leq \tilde{r}^{-2};$

(c) $\delta(t)\leq \tilde{\delta}_{j} $ on
$[j\varepsilon^2,(j+1)\varepsilon^2],$ for $0 \leq j\leq l-1,$ and
$\delta(t)\leq \tilde{\delta}_{l} $ on $[(l-1)\varepsilon^2,T').$

\end{lemma}

\begin{proof}
Suppose $R(\cdot,\cdot)\leq r_0^{-2}$ on $
P(x_0,t_0,r_0,-r_0^2)=\{ (x',t') \ |\ x'\in B_{t'}(x_0,r), t' \in
[t_0 -r^2,t_0] \},$  we will estimate
$vol_{t_0}(B_{t_0}(x_0,r_0))/ r_0^4$ from below.

Step 1: In this step, we deal with the estimates on scales not too
small comparing with $\tilde{r}.$ We assume $r_0\geq
\frac{\tilde{r}}{C(\varepsilon)},$ where $C(\varepsilon)$ is some
fixed constant (to be determined later) depending only on
$\varepsilon.$ In this case, we adapt the proof of Lemma 5.5 in
\cite{CZ05F} as follows.

Since the surgeries occur in place where the curvature is bigger
than $\delta^{-2}\tilde{r}^{-2},$ which is much larger than
$\tilde{r}^{-2},$ we first modify the argument of Lemma 5.5 in
\cite{CZ05F} to show any $\mathcal{L}$ geodesic $\gamma(\tau), \tau
\in[0,\bar{\tau})$ ($\bar{\tau}\leq t_0-(l-1)\varepsilon^2$),
starting from $(x_0,t_0)$ with reduced length $\leq
\varepsilon^{-1}$, stays far away from the place where surgeries
occur. More precisely, we claim that if some $\gamma(\tau_0)$ is not
far from some cap which is glued by surgery procedure at time
$t=t_0-\tau_0,$ then the reduced length of $\gamma$ defined by
$$
\frac{1}{2\sqrt{\bar{\tau}}}\int_{0}^{\bar{\tau}}\sqrt{\tau}(R(\gamma(\tau),\tau)+|\dot{\gamma}(\tau)|^2)d\tau
$$
is $\geq 25\varepsilon^{-1}$.

This estimate for manifold case was established in (5.8) on page 238
of \cite{CZ05F}. Let us recall the proof of this estimate for the
manifold case given in \cite{CZ05F}. Note that the place performed
$\delta-$ cutoff surgery is deeply inside the horn under
normalization and the parabolic region $P(x_0,t_0,r_0,-r_0^2)$ is
far from it by curvature estimates for canonical neighborhoods. Thus
at the time $t=t_0-\tau_0$, the point $\gamma(\tau_0)$ lies deeply
inside a very long tube and the segment $\gamma(\tau), \tau \in
[0,\tau_0]$, tends to escape from the tube. If $\gamma(\tau)$
escapes from the very long tube within short time $\leq
CR(x_1,t_0-\tau_0)^{-1}$ from $\tau_0,$ where $C$ is some universal
constant and $x_1$ is a point in the neck where surgery takes place,
then $ \int_{0}^{\bar{\tau}}|\dot{\gamma}(\tau)|^2d\tau $
contributes a big quantity to the above integral since the tube is
quite long. However if $\gamma(\tau)$ stays a while $\geq
CR(x_1,t_0-\tau_0)^{-1}$ on the long tube, then
$\int_{0}^{\bar{\tau}}R d\tau $ contributes a large quantity to the
above integral,  since for any $1>\zeta>0,$ we have the estimate
$$R(x,t)\geq R(x_1,t_0-\tau_0)
\frac{Const.}{\frac{n-1}{2}-R(x_1,t_0-\tau_0)(t-t_0+\tau_0)}$$ on
$\gamma\mid_{[\tau_0-\frac{n-1}{2}(1-\zeta)R(x_1,t_0), \tau_0]},$
when  $\delta$ is small enough and $\gamma(\tau)$ stays not far from
the cap.

All the above arguments of \cite{CZ05F} still work in our present
orbifold case except the verification of the last statement on the
estimate of the scalar curvature on the tube. In \cite{CZ05F}, the
proof of the above estimate on the scalar curvature on the tube was
given as follows. Recale the solution with factor
$R(x_1,t_0-\tau_0)$ around $(x_0,t_0-\tau_0 )$. Since the necks in
the manifold case of \cite{CZ05F} are not collapsed, we can extract
a convergent limit as $\delta\rightarrow \infty.$ The limit, called
standard solution,  is rotationally symmetric, exists exactly on the
time interval $[0,\frac{n-1}{2})$ and has curvature estimates
$\frac{Const.}{\frac{n-1}{2}-s}$ at time $s.$ But in the current
orbifold case, at a priori, we do not know whether the necks in the
canonical neighborhoods are collapsed or not. Our new argument is to
use Hamilton's canonical parametrization for (the part of) horn:
$\Phi: \mathbb{S}^3\times (-L,L) $ such that the surgery is taken
place on $[0,4),$ and there is finite group $\Gamma$ of global
isometric actions of $\mathbb{S}^3,$ such that $\Phi$ is $\Gamma$
invariant, and $\Phi: \mathbb{S}^3/\Gamma\times (-L,L)$ is
diffeomorphic to its image and each $\mathbb{S}^3/\Gamma$ is mapped
to a constant mean curvature hypersurface.  Moreover the pull back
metric on $\mathbb{S}^3\times (-L,L)$ (after scaling) is very close
to the standard cylinder. We perform a standard surgery on $
\mathbb{S}^3\times (-L,L)$ by cutting open the neck and glue back a
cap, denote the resulting space by $Y$.  Clearly, we can require
 $\Phi$ to
be extended and  defined on $Y$ to the space after surgery, and the
pull back metric is close to the standard capped infinite cylinder.
We pull back the solution also to $Y.$  Note that the gradient
estimate in the canonical neighborhood assumption implies a
curvature bound for the solutions. Then as $\delta\rightarrow 0,$ we
can apply the uniqueness theorem \cite{CZ05U} to show that the
solutions on $Y$ around point near the cap converge to a standard
solution. So the above estimate on the scalar curvature also holds
in our present case.

After proving that any $\mathcal{L}$ geodesic of  reduced length $<
25 \varepsilon^{-1}$ does not touch the surgery region, one can
apply the same argument of Lemma 5.5 in \cite{CZ05F} of using
Perelman's Jacobian comparison to bound
$vol_{t_0}(B_{t_0}(x_0,r_0))/ r_0^4$ from below by constant
depending only on $\varepsilon,\kappa_{l-1},r_{l-1}$ ( see
\cite{CZ05F}, pages 238-241, for the details).

 Step 2: In this step, we deal with the estimates
on scales less than $\frac{\tilde{r}}{C(\varepsilon)}.$ This case is
easier in \cite{CZ05F} because the space has no singularities and
the canonical neighborhoods are not collapsed there. In our present
orbifold case, at a priori, the canonical neighborhoods in our
definitions may be sufficiently collapsed. So we need a new
argument.

Clearly, we may assume $R(x',t')= r_0^{-2}$ for some point on $
P(x_0,t_0,r_0,-r_0^2)=\{ (x',t') \ |\ x'\in B_{t'}(x_0,r), t' \in
[t_0 -r^2,t_0] \},$ otherwise we enlarge $r_0.$ Since $r_0\leq
\frac{\tilde{r}}{C(\varepsilon)},$ by the definition of canonical
neighborhoods, we can choose $C(\varepsilon)$ large enough so that
every point in $B_{t_0}(x_0,r_0)$ has curvature $\geq
\tilde{r}^{-2}$. In particular, the point $x_0$ at the time $t_0$
has a canonical neighborhood, which is a strong $\varepsilon$-neck
or $\varepsilon-$cap. For both cases, the canonical neighborhood
contain an $\varepsilon$-neck $N$ which is close to
$(-\varepsilon^{-1},\varepsilon^{-1})\times (\mathbb{S}^3/\Gamma).$
Clearly, in order to get the $\kappa-$noncollapsing, we only need to
bound the order $|\Gamma|$ of the group $\Gamma$ from above.

Now we consider one of the boundaries $\partial N$ of $ N.$ Since
the curvature is $\geq\tilde{r}^{-2}$ there, there is an
$\varepsilon-$neck or $\varepsilon-$cap adjacent to $N.$ If it is
the $\varepsilon-$ cap adjacent to $N,$ we stop for this end and
consider the other boundary of $N.$ If it is a $\varepsilon-$ neck
 adjacent (denoted by $N'$) to $N,$ and $N'$ contains a point having curvature $\leq C(\varepsilon)^2\tilde{r}^{-2},$
 then we also stop. Otherwise,  $ N\cup N'$ form a longer (topological)
 neck, we consider the boundary of $N'$ and continue the argument.
 We do the same argument for the another boundary of $N.$
 Since there is a point $\bar{x}$ on space such
that $R(\bar{x},t_0)\leq \tilde{r}^{-2}$ by assumption (b), there
must be an extension of one boundary of $N$ such that the final
adjacent neck or cap having a point with curvature $\leq
C(\varepsilon)^2\tilde{r}^{-2}.$ By canonical neighborhood
assumption, the curvature at the final neck or cap are $\leq
C(\varepsilon)^2\tilde{r}^{-2}.$ We conclude that there is a tube
$T$ consisting of $\varepsilon-$ necks such that $T$ contains the
initial neck $N$ and another $\varepsilon-$neck $N_1$ where the
curvatures are $\leq C(\varepsilon)^2 \tilde{r}^{-2}.$ By step 1, we
can bound
\begin{equation} \label{noncollapse}\frac{vol_{t_0}(N_1)}{\varepsilon^3
diam(N_1)^4} \geq \frac{1}{C(\varepsilon,\kappa_{k-1},r_{l-1})}
\end{equation}from below uniformly. By using Hamilton's canonical
parametrization $\Phi: \mathbb{S}^3\times (A,B)$ to parametrize
$T$, $\Gamma$ acts isometrically on the factor $\mathbb{S}^3$ on
the whole $ \mathbb{S}^3\times (A,B).$ This gives
$|\Gamma|vol_{t_0}(N_1) \leq C(\varepsilon) diam(N_1)^4.$ By
combining with (\ref{noncollapse}), we get a uniform upper bound
of $|\Gamma|.$

The proof of the theorem is completed.

\end{proof}

\begin{theorem} Given a compact four-dimensional orbifold $(X,g)$
with positive isotropic curvature and with at most isolated
singularities. Given any fixed small constant $\varepsilon > 0.$ one
can find three non-increasing positive and continuous functions
$\widetilde{\delta}(t)$, $\widetilde{r}(t)$ and
$\widetilde{\kappa}(t)$ defined on whole $[0,+\infty)$ with the
following properties. For arbitrarily given positive continuous
function $\delta(t)\leq \widetilde{\delta}(t)$ on $[0,+\infty)$, the
Ricci flow with $\delta(t)-$cutoff surgery, starting with $g$,
admits a solution satisfying the a priori assumption (with accuracy
$\varepsilon$ with $r=\widetilde{r}(t)$ ) and $\kappa-$
noncollapsing (with $\kappa=\widetilde{\kappa}(t)$) on a maximal
time interval $[0,T)$ with $T<+\infty$ and becoming extinct at $T.$
Moreover, the solution is obtained by performing at most finite
number of $\delta-$cutoff surgeries on $[0,T)$.
\end{theorem}
\begin{proof}
The pinching assumption is justified in Lemma \ref{jup}. To justify
the canonical neighborhood assumption, we can apply the same
argument as in manifold case, because we have all ingredients we
need to mimic the proof of Proposition 5.4 in \cite{CZ05F}.  We note
the surgery does not occur on the place where the scalar curvature
achieves its minimum. Then by applying the maximum principle to the
scalar curvature equation $(\frac{\partial}{\partial
t}-\triangle)R=2|Ric|^2$, we conclude that the surgical solution
must be extinct in finit time. To prove the finiteness of the number
of surgeries, we need to check the $\kappa-$ noncollapsing for the
solution. In fact, the $\kappa-$noncollapsing follows from Lemma
\ref{kappan}. Therefore, the proof of the theorem is completed.

\end{proof}

\subsection{Recovering the topology}
\begin{proof} of Theorem \ref{t2}.

Consider a surgical solution, obtained by the previous theorem, to
the Ricci flow with surgery on a maximal time interval $[0,T)$ with
$T<+\infty$. Now we can recover the topology of the initial orbifold
as follows.

Suppose our surgeries times are $0<t_1<t_2<\cdots, t_k< T.$   For a
surgery time ${t_p}^{+},$ after surgeries, denote by $M_1^p, M_2^p,
\cdots M_{i_p}^{p}$ the all connected components either containing
no points of $\Omega_{\rho(t)}$ or having positive curvature
operator. The rest connected components are denoted by
$N^{p}_{1},\cdots, N_{i_p'}^{p}.$ Recall that our construction for
the surgical solution is to stop the Ricci flow on those $M^p_{l}$
for $l=1,\cdots, i_p$ and to continue the Ricci flow on $N^{p}_{l}$
for $ l=1,\cdots, i_p'.$ Note that all connected components at time
$T^{-}$ either contain no points of $\Omega_{\rho(t)}$ or have
positive curvature operator. We denote them by $M_1^{k+1},
M_2^{k+1}, \cdots, M_{i_{k+1}}^{k+1},$ they are actually
$N^{k}_{1},\cdots, N_{i_k'}^{k}.$ We collect all these $M^{i}_j$'s
in a set $\mathcal{S}=\{M_1^1,\cdots, M_{i_{k+1}}^{k+1}\}.$ For each
$M_{j}^{i},$ we will mark a finite number of points $P^{i}_{j,l},$
$l=1,2,\cdots i^{j},$ in the following inductive way.

At first surgery time $t_1,$ we perform a cut-off surgery along a
$\delta-$horn $H$, i.e. we cut open $\delta-$horn along a neck $N$
and glue back a cap or orbifold cap  to the finite part of the horn
connected to $\Omega_{\rho}.$  Remember we also glue back a cap or
orbifold cap
 to the infinite part of the horn (so called
horn-shape end). We  denote the tips of these two caps by  $P$ and
$\bar{P},$ denote these two caps by $C^{P}$ and $C^{\bar{P}}$
respectively.  Through the neck $N$ we cut open, the surgery
procedure establishes a diffeomorphim $\varphi_{P\bar{P}}'$ from
$\partial C^{P}$ to $\partial C^{\bar{P}}.$ Let $S_P$ and
$S_{\bar{P}}$ be the unit tangent spheres at $P$ and $\bar{P},$
then $\varphi_{P\bar{P}}'$ induces an isotopic class
$\varphi_{P\bar{P}}$ of diffeomorphism from $S_P$ to
$S_{\bar{P}}.$  We assign the pairs $(P,S_P)$ and
$(\bar{P},S_{\bar{P}})$ to the manifolds or orbifolds where they
are located.  Inductively, at surgery time $t_p^{-},$ for each
$N^{p-1}$ with some points already been marked by the previous
steps,  we leave these marked points alone, and add new points
produced by performing surgeries at $t_p$ on $N^{p-1}.$ Note that
the previous marked points may be separated to lie in different
connected components after surgeries. Once a component $M_j^{i}$
is terminated at a surgery time $t_j,$ then there is no more
points assigned to it in any later surgery times. We collect all
these marked  points $(P,S_P),$ $(\bar{P},S_{\bar{P}})$ and
isotopic classes $\varphi_{P\bar{P}}$ together.

Now we investigate the topology of each $M^{i}_j\in \mathcal{S}.$
We know at time $t_i^{+},$  $M^{i}_j$ is either diffeomorphic to a
spherical orbifold $\mathbb{S}^4/\Gamma$  with at most isolated
singularities, or it is  covered by $\varepsilon-$ necks and
$\varepsilon-$ caps. Now we consider the latter case.

If $M^i_j$ contains no caps, then $M^{i}_j$ is diffeomorphic
 to smooth manifold $\mathbb{S}^3/ \Gamma\times \mathbb{S}^1$ or $\mathbb{S}^3/ \Gamma
 {\times}_{f}
 \mathbb{S}^1.$

 If $M^{i}_j$ contains  caps, then $M^{i}_j$ is diffeomorphic to  either  smooth
 manifold $\mathbb{S}^4,$ $\mathbb{R}\mathbb{P}^4,$
 $C_{\Gamma}^{\sigma}\cup_f
  C_{\Gamma'}^{\sigma'},$ or one of the orbifolds
  $C_{\Gamma}^{\sigma}\cup_f
  C_{\Gamma'},$ $C_{\Gamma}\cup_f C_{\Gamma'},$ $\mathbb{S}^4/(x,\pm
  x'),$ $\mathbb{S}^4/(x,\pm x')\#\mathbb{R}\mathbb{P}^4,$ $\mathbb{S}^4/(x,\pm x')\# \mathbb{S}^4/(x,\pm x').$

  So we conclude that  each $M^{i}_j$ is diffeomorphic to a
  connected sum of at most two spherical
  orbifolds
  $\mathbb{S}^4/\Gamma_{j,1}^i$ and $\mathbb{S}^4/\Gamma_{j,2}^i.$
   For each $(P,S_P),$ $(\bar{P},S_{\bar{P}})$ and
 $\varphi_{P\bar{P}},$  we know reversing the
surgery procedure is to do connected sum by removing the pair of
points $P,$ $\bar{P}$ and using $\varphi_{P\bar{P}}$ to identify
the boundaries.
   Therefore, our original orbifold is
diffeomorphic to the connected sum of spherical orbifolds
$\mathbb{S}^4/\Gamma$ with at most isolated singularities. This
completes the proof of Theorem \ref{t2}.
\end{proof}

\section{Proof of Main Theorem} \label{5}

 The main purpose
of this section is to deduce the Main Theorem  from Theorem
\ref{t2}.   We need several lemmas on the group actions on the
sphere $\mathbb{S}^n.$

\begin{lemma}\label{lemma5.1}Let $G\subset SO(2n+1)$($n\geq 2$) be a finite
subgroup such that each nontrivial element in $G$ has exactly one
eigenvalue equal to $1.$  Then there is a common nonzero vector
$0\neq v\in \mathbb{R}^{2n+1}$ such that  for all $g\in G$ we have
$g(v)=v.$
\end{lemma}
\begin{proof}
The idea of the proof is similar to the classification of fixed
point free finite subgroups of the isometry group of
${\mathbb{S}}^{2n+1}$ in \cite{Wolf}. We divide our argument into
two cases.

Case i): $|G|$ is even. In this case, there is an element of order 2
by Cauchy theorem. We denote this element by $\sigma.$ We claim that
$\sigma$ is the unique element of order 2 in $G.$ Indeed, suppose
$\sigma'$ is another distinct order 2 element. Note that, by our
assumption, $\sigma$ and $\sigma'$ must have one eigenvalue equal to
1 and $2n$ eigenvalues equal to $-1.$ Let $E_1$ and $E_2$ be the
eigenspaces with eigenvalue $-1$ of $\sigma$ and $\sigma'$
respectively. Clearly $\sigma\sigma'^{-1}=1$ on $E_1\cap E_2.$ Since
$n\geq 2,$ the intersection $E_1\cap E_2$ has dimension $\geq
2n-1\geq 3,$ this implies that $\sigma=\sigma'$ on the whole space.
This is a contradiction.

By the uniqueness of $\sigma,$ we know that $g^{-1}\sigma g=\sigma$
for any $g\in G.$ Suppose $\sigma (v)=v$ for $|v|=1,$ then $\sigma
g(v)=g(v).$ Hence $g(v)=v$ or $g(v)=-v.$ We claim that $g(v)=-v$
cannot happen. The reason is as follows. Let $g(u)=u$ for $|u|=1$,
then $g^2(u)=u.$ By combining with $g^2 (v)=v,$ we know that either
$g^2=1$ or $v=\pm u.$ If $g(v)=-v,$ then $v$ cannot be $\pm u,$ so
$g$ has order 2, and equal to $\sigma$ by the uniqueness of order 2
element, this contradicts with $\sigma (v)=v$. So we have showed
that $g(v)=v$ for any $g\in G.$

Case ii): $|G|$ is odd. First, we show that every subgroup of order
$p^2$ ($p$ is a prime number)
 of $G$ is cyclic. Namely, we will show that $G$ satisfies
the $p^2$ condition.

Indeed, suppose $H$ is a noncyclic subgroup of order $p^2$ for some
prime number $p.$ Since a group of order $p^2$ with $p$ prime must
be abelian, we can apply the same argument as case (i) to conclude
that there is a unit vector $v$ fixed by the whole group. Let $W
\cong\mathbb{R}^{2n}$ be the orthogonal complement of $v$ in
$\mathbb{R}^{2n+1}.$ Then $H$ induces a fixed point free action on
the unit sphere $\mathbb{S}^{2n-1}$ of $W.$ So for any $v'\in
\mathbb{S}^{2n-1},$ we have $0=\sum_{g\in H} g(v'). $ On the other
hand, since $G$ is abelian and noncyclic, we conclude that each
nontrivial element has order exactly $p,$  the intersection of any
two distinct order $p$ groups contains only the identity. Let $H_i,$
$i=1,\cdots, m,$($m\geq 2$) be the subgroups in $H$ of order $p$,
then for any $v'\in \mathbb{S}^{2n-1},$
$$0=\sum_{g\in G} g(v')=\sum_{i}^{m}\sum_{g\in
H_i}g(v')-(m-1)v'=-(m-1)v', $$ where we have used the fact
$\sum_{g\in H_i}gv'=0$ since $H_i$ also acts freely on
$\mathbb{S}^{2n-1}.$ The contradiction shows that $H$ is cyclic.

The fact that $G$ satisfies $p^2$ condition implies that every Sylow
subgroup of $G$ is cyclic (see Theorem 5.3.2 in \cite{Wolf}, note
that since $|G|$ is odd, so must be $p$). By Burnside theorem (see
Theorem 5.4.1 in \cite{Wolf}) , once we know that every Sylow
subgroup of $G$ is cyclic, then $G$ is generated by two elements $A$
and $B$ with defining relations
\begin{equation*}
\begin{split}
& A^m=B^n=1,\ \  BAB^{-1}=A^{r}, \ \ |G|=mn;
\\
& ((r-1)n,m)=1, \ \ \  r^{n}\equiv 1 (mod\ \  m).
\end{split}
\end{equation*}
Let $A(v) = v$ for $|v|=1$. We will show that $B(v)=v.$ Indeed, by
the relation $BAB^{-1}=A^{r},$ we have $AB^{-1}(v)=B^{-1}(v).$ This
implies $B^{-1}(v)= v$ or $B^{-1}(v)= -v.$ $B^{-1}(v)= -v$ will not
happen, because it implies $B^{-2}=1$ by the argument in case i).
This will imply the order of the group is even, which is a
contradiction with our assumption.  So $v$ is fixed by the whole
group $G.$

\end{proof}
\begin{lemma} \label{lemma5.2}Let $G\subset O(2n+1)$($n\geq 2$) be a finite group
of orthogonal matrices such  that each nontrivial element in $G$ has
at most one eigenvalue equal to $1.$  Then there is a finite group
$G'\subset SO(2n)$ acting freely on the sphere $\mathbb{S}^{2n-1}$
and a character $\chi:G'\rightarrow \{\pm1\}$ such that after
conjugation, the group $G= \{ \left(
  \begin{array}{cc}
    \chi(g) & 0\\[1mm]
    0 & g
   \end{array}
\right): g\in G'\} .$
\end{lemma}

\begin{proof}
Let $G_0=G\cap SO(2n+1).$ If $G_0=G,$ then from Lemma \ref{lemma5.1}
we are done by choosing $G' =$ the restriction of $G$ on the
orthogonal complement of $v$ (the common unit vector fixed by $G$)
and $\chi\equiv 1.$ If $G_0\neq G,$ then $G_0$ is an index 2 normal
subgroup of $G$. Since an element of $G_0$ must has 1 as its
eigenvalue, it has exactly one eigenvalue 1 by our assumption. By
Lemma \ref{lemma5.1}, there is a common unit vector $v$ fixed by
$G_0.$ For any $g\in G\backslash G_0,$ we claim $g(v)=-v.$ The
argument is as follows.

Since $g^2\in G_0,$ we have $g^2(v)=v.$ Let $E=\text{span}\{ v,
g(v)\}.$ We will show $\dim E=1.$  Indeed, suppose $\dim E=2.$ Since
$g(v+g(v))=v+g(v)$ and $g(v-g(v))=-(v-g(v)),$
$E$ is an invariant subspace of $g$, $\dim E^{\perp}$ is odd and
$\det(g \mid_{E^{\perp}})=1$. So $g$ has another fixed nonzero
vector in $E^{\perp}.$ This contradiction shows that $g(v)=v$ or
$g(v)=-v.$ If $g(v)=v,$ then $\det(g\mid_{\{v\}^{\perp}})=-1,$ and
hence $g$ must have another fixed vector in $\{v\}^{\perp}$ since
$\dim \{v\}^{\perp}$ is even, this again contradicts with the
assumption that $g$ has at most one eigenvalue 1. This proves our
claim.

Next, we show that $G$  acts freely on the unit sphere of
$\{v\}^{\perp}$. For this, we only need to check for any $g\in
G\backslash G_0,$ $g$ has no nonzero fixed vector in
$\{v\}^{\perp}.$ But if this is not true, we have $g^2=1,$ this
implies that $g$ has one eigenvalue 1 (by assumption) and $2n$
eigenvalues $-1,$ which contradicts with $\det(g)=-1.$ To finish the
proof, we only have to take $G'=$ the restriction of $G$ on
$\{v\}^{\perp}$ and $\chi$ is the character which takes value 1 on
$G_0$ and -1 otherwise.
\end{proof}
In the following, we prove Main Theorem by using  Theorem \ref{t2}
and Lemmas \ref{lemma5.1}and \ref{lemma5.2}.
\begin{proof} of Main Theorem.
With the help of the above lemmas, we can now describe the structure
of the spherical orbifolds $\mathbb{S}^{4}/\Gamma$ appearing in
Theorem \ref{t2}. Since the resulting quotient space
$\mathbb{S}^{4}/\Gamma$ has at most isolated singularities, each
element of $\Gamma$ has at most a pair of antipodal fixed points, so
the group $\Gamma$ satisfies the assumptions in Lemma
\ref{lemma5.2}. There are three cases for $\mathbb{S}^{4}/\Gamma.$
The first case is $\Gamma$ acts on $\mathbb{S}^{4}$ freely, the
resulting space is a smooth manifold diffeomorphic to $\mathbb{S}^4$
or $\mathbb{R}\mathbb{P}^4.$ The second case is $\Gamma\neq \{1\}$
and $\Gamma\subset SO(5).$ Assume $\mathbb{S}^{4}\subset
\mathbb{R}^5$ has equation $x_1^2+x_2^2+\cdots+x_5^5=1.$ By Lemma
\ref{lemma5.2},we may assume the north pole $P=(0,0,0,0,1)$ and
south pole $-P=(0,0,0,0,-1)$ of $\mathbb{S}^{4}$ are the fixed
points of $\Gamma.$ Let $\mathbb{S}^3=\mathbb{S}^{4}\cap \{x_5=0\},$
the group action fixes $x_5$ coordinate, so
$\mathbb{S}^{4}/\Gamma\backslash \{P,-P\}$ is diffeomorphic to
$\mathbb{S}^{3}/\Gamma\times (0,1).$ The third case is
$\Gamma_0=\Gamma\cap SO(5)\subsetneqq \Gamma.$ By Lemma
\ref{lemma5.2},we may also assume the north pole $P$ and south pole
$-P$ of $\mathbb{S}^{4}$ are the fixed points of $\Gamma_0.$ The
group $\Gamma$ can act on $\mathbb{S}^3\times \mathbb{R}$ in a
natural way, and we can equip a metric on
$\mathbb{S}^{4}/\Gamma\backslash \{P\},$ which is locally isometric
to $\mathbb{S}^3\times \mathbb{R}$ and the manifold has only one end
which is isometric to $\mathbb{S}^3/\Gamma_0\times [0,\infty).$

Let $X_1,\cdots, X_l$ be the orbifolds appearing in Theorem
\ref{t2}.  The orbifold connected sum procedure can be described in
 two steps, the first step is to resolve all singularities of $X_1, \cdots, X_l$
 which appear pairwise in the surgery procedures of the Ricci flow
 by orbifold connected sums,  the resulting spaces consists of
 finite number of smooth closed manifolds, denoted by $Y_1, \cdots
 Y_k.$ The next step is to perform the orbifold connected sums among these \textbf{manifolds}
 $Y_1, \cdots, Y_k$ and a finite number of $\mathbb{S}^4,$ $\mathbb{RP}^4.$
 Now we investigate the topology of each components $Y_i$ in the
 first step mentioned above. We remove all singularities from all $X_j,$
 the orbifolds falling in the second case give us necks $\mathbb{S}^3/\Gamma\times (0,1),$ the orbifolds in the third case
  give us caps of the form $C_{\Gamma_0}^{\sigma}.$  Since in the first
  step, each end of one such neck has to be joined with a cap, or another
  neck, producing a longer cap or neck; the end of each cap has to
  be joined with a neck or another cap.
So each $Y_i$ is
  diffeomorphic to either  $\mathbb{S}^4,$ $\mathbb{R}\mathbb{P}^4$
  or the manifold (denoted by $\mathbb{S}^3/\Gamma\times
[0,1]/f$ temporarily) obtained by gluing the boundaries of
$\mathbb{S}^3/\Gamma\times [0,1]$ by utilizing a diffeomorphism
$f:\mathbb{S}^3/\Gamma\times \{0\}\rightarrow
\mathbb{S}^3/\Gamma\times \{1\},$   or the manifold
$C_{\Gamma}^{\sigma}\cup_{f''} C_{\Gamma'}^{\sigma''}$ obtained by
gluing the boundaries  $\partial C_{\Gamma}^{\sigma}$ and $
\partial C_{\Gamma'}^{\sigma''}$ by a diffeomorphim $f''.$ It is known that two spherical three space forms are isometric if they
are diffeomorphic. So in the manifold $C_{\Gamma}^{\sigma}\cup_{f''}
C_{\Gamma'}^{\sigma''},$ the group $\Gamma$ and $\Gamma'$ are
conjugate (in $O(4)$). After a conjugation, we have
$\Gamma=\Gamma',$ and assume $C_{\Gamma}^{\sigma}\cup_{f''}
C_{\Gamma'}^{\sigma'}\cong C_{\Gamma}^{\sigma}\cup_{f'}
C_{\Gamma}^{\sigma'}.$ Since the diffeomorphism types of
$\mathbb{S}^3/\Gamma\times [0,1]/f$ and
$C_{\Gamma}^{\sigma}\cup_{f'} C_{\Gamma}^{\sigma'}$ remain unchanged
if we deform $f$ and $f'$ isotopically. Moreover, by\cite{Mc}, the
diffeomorphims $f,f'\in Diff(\mathbb{S}^3/\Gamma)$ are isotopic to
 isometries $I_f, I_f'\in ISO(\mathbb{S}^3/\Gamma).$
  So if we equip $\mathbb{S}^3/\Gamma\times [0,1]$ with the standard product metric,
   the induced metric on $\mathbb{S}^3/\Gamma\times
[0,1]/I_f$ is locally isometric to $\mathbb{S}^3\times
\mathbb{R}.$ Similarly, if we equip $C_{\Gamma}^{\sigma}$ and
$C_{\Gamma}^{\sigma'}$ the standard metric which is locally
isometric to $\mathbb{S}^3\times \mathbb{R}$ and the end is
isometric to product $\mathbb{S}^3/\Gamma\times [0,1),$ then the
induced metric on  $C_{\Gamma}^{\sigma}\cup_{I_f'}
C_{\Gamma}^{\sigma'}$ is locally isometric to
$\mathbb{S}^3/\Gamma\times \mathbb{R}.$  In both cases, the
universal covers are $\mathbb{S}^3\times \mathbb{R}.$ Now we have
shown that each $Y_i$ is diffeomorphic to either $\mathbb{S}^4,$
or $\mathbb{R}\mathbb{P}^4$ or
 $\mathbb{S}^3\times
\mathbb{R}/G,$ where $G$ is a fixed point free cocompact discrete
subgroup of the isometries of standard metric on
$\mathbb{S}^3\times \mathbb{R}.$ Therefore, the manifold $M$ is
diffeomorphic to an orbifold connected sum of $\mathbb{S}^4,$ or
$\mathbb{R}\mathbb{P}^4$ or
 $\mathbb{S}^3\times
\mathbb{R}/G.$ Note that doing orbifold connected sum through two
embedded 3-spheres on a connected smooth  manifold is equivalent
to doing the usual connected sum of this manifold with
$\mathbb{S}^3\times \mathbb{S}^1$ or
$\mathbb{S}^3\tilde{\times}\mathbb{S}^1,$ also, doing orbifold
connected sum between two connected smooth manifolds is just doing
the usual connected sum by suitably choosing the orientations of
the embedded 3-spheres. Therefore, we have shown the manifold $M$
is diffeomorphic to the usual connected sum of $\mathbb{S}^4,$ or
$\mathbb{R}\mathbb{P}^4$ or
 $\mathbb{S}^3\times
\mathbb{R}/G.$

\end{proof}

  \begin{corollary} \label{cor5.1}A compact 4-orbifold with at most isolated singularities  with
positive isotropic curvature is diffeomorphic to the connected sum
 $\#_{i}(\mathbb{S}^3\times
\mathbb{R}/G_i)\#_j (\mathbb{S}^4/\Gamma_j),$ where $G_i$ and
$\Gamma_j$ are standard group actions,  the connected sum is in
the usual sense.
  \end{corollary}
\begin{proof}
By using the same proof of the Main Theorem, we only need to
consider those components which are diffeomorphic to either
$C_{\Gamma}^{\sigma}\cup_f C_{\Gamma},$ or $C_{\Gamma}\cup_f
C_{\Gamma}.$ Note that by \cite{Mc}, $f$ is isotopic to an
isometry $f'$ of $\mathbb{S}^3/\Gamma,$ which can be naturally
extended to a diffeomorphism of $C_{\Gamma}$ to itself.  This
gives a diffeomorphism from $\mathbb{S}^4/\{\Gamma,\hat{\sigma}\}$
or $\mathbb{S}^4/{\Gamma}$ to $C_{\Gamma}^{\sigma}\cup_{f'}
C_{\Gamma},$ or $C_{\Gamma}\cup_{f'} C_{\Gamma}.$
\end{proof}

\section{Appendix} Let $\varepsilon$ be a positive constant.
We call an open subset $N\subset X$ in an metric space \textbf{GH}
\textbf{$\varepsilon$-neck of radius $r$} if $r^{-1}N$ is
homeomorphic and Gromov-Hausdorff  $\varepsilon$-close to a neck
$S \times \mathbb{I}$ where $S$ is some Alexandrov space with
nonnegative curvature and without boundary and $diam(S)\leq
\frac{1}{\sqrt{\varepsilon}}$ and
$\mathbb{I}=(-\varepsilon^{-1},\varepsilon^{-1}).$ \vskip 0.3cm
  \begin{proposition}\label{p5.1}
    There exists a constant $\varepsilon_0=\varepsilon_0(n)>0$ such that for any
complete noncompact  $n-$ dimensional intrinsic  Alexandorv space
$X$ with nonnegative curvature,  there is a positive constant
$r_0>0$ and a compact set $K\subset X$ such that any GH
$\varepsilon$-neck of radius $r\leq r_0$ on $X$ with $\varepsilon
\leq \varepsilon_0$  must be contained in $K$ entirely.
  \end{proposition}
\begin{proof} When the space is smooth manifold, and the topology defining the $\varepsilon-$ neck
is in $C^{[\frac{1}{\varepsilon}]},$ the proof is given by
\cite{CZ05F}. Now we modify the arguments to the present
situation, the key observation is that we essentially  used only
triangle comparison in \cite{CZ05F}. Here we include the proof for
completeness.

  We argue by contradiction. Suppose there exists a
sequence of positive constants $\varepsilon^\alpha \rightarrow 0$
and a sequence of $n$-dimensional complete noncompact pointed
Alexandrov space $(X^\alpha,P^{\alpha})$ with nonnegative
curvature such that for each fixed $\alpha$, there exists a
sequence of GH $\varepsilon^\alpha$-necks $N_{k}$ of radius
$r_k\leq 1/k$ on $X^\alpha,$ and $N_{k}\subset X\backslash
B(P^{\alpha},k).$ Recall that by the definition of
Gromov-Hausdorff distance, there is metric space $Z_k$ containing
isometric embedding s of $r_k^{-1}N_{k}$ and $S \times \mathbb{I}$
such that $S \times \mathbb{I}\subset
B_{\varepsilon^{\alpha}}(r_kN_{k})$ and $r_k^{-1}N_{k} \subset
B_{\varepsilon^{\alpha}}(S \times \mathbb{I}).$ Let $P_{k}\in
r_k^{-1}N_k$ be a point having distance $\leq
\varepsilon^{\alpha}$ with $S \times \{0\}$ (in $Z_k$). Then we
have $d(P^{\alpha},P_{k})\rightarrow \infty$ as $k\rightarrow
\infty.$

 Let $\alpha$ to be fixed and
sufficiently large. Connecting each $P_k$ to $P^{\alpha}$ by a
minimizing geodesic $\gamma_k,$ passing to subsequence, we may
assume the angle $\theta_{kl}$ between geodesic $\gamma_k$ and
$\gamma_l$ at $P^{\alpha}$ is very small and tends to zero as $k,
l\rightarrow+\infty$, and the length of $\gamma_{k+1}$ is much
bigger than the length of $\gamma_k$.  Let us connect $P_k$ to
$P_l$ by a minimizing geodesic $\eta_{kl}$.

For any three points $A,B,C\in X,$ we use $\bar{\Delta}
_{\bar{A}\bar{B}\bar{C}}$ to denote corresponding triangle in
plane $\mathbb{P}$ with
$d(A,B)=|\bar{A}\bar{B}|,d(A,C)=|\bar{A}\bar{C}|,d(B,C)=|\bar{B}\bar{C}|,$
and we also use $\bar{\angle}\bar{A}\bar{B}\bar{C}$ to denote the
angle of $\bar{\Delta} _{\bar{A}\bar{B}\bar{C}}$ at $\bar{B}.$

Clearly,$\bar{\angle}\bar{P^{\alpha}}\bar{P_{k}}\bar{P_l}$ is
close to $\pi$ by comparison. Let $P_{k}'\in \gamma_k \cap
\partial N_k$ and $P_{k}''\in \eta_{kl} \cap
\partial N_k$ then it is clear for any point $x\in \partial
N_k,$ we have either $\bar{\angle}\bar{P_k'}\bar{P_k}\bar{x}$ is
small and $\bar{\angle}\bar{P_k''}\bar{P_k}\bar{x}$ is close to
$\pi,$ or $\bar{\angle}\bar{P_k'}\bar{P_k}\bar{x}$ is close to
$\pi$ and $\bar{\angle}\bar{P_k''}\bar{P_k}\bar{x}$ is small. This
depends on $\bar{x}$ lies which  connected component of $\partial
N_k.$

 By using the above facts and triangle comparison (see \cite{CZ05F}),  we can show that as $k$ large
enough, each minimizing geodesic $\gamma_l$ with $l>k$, connecting
$P^{\alpha}$ to $P_l$, must go through the whole $N_k$.

Hence by taking a limit, we get a geodesic ray $\gamma$ emanating
from $P$ which passes through all the necks $N_k$, $k = 1, 2,
\cdots,$  except a finite number of them. Throwing these finite
number of necks, we may assume $\gamma$ passes through all necks
$N_k$, $k=1,2,\cdots.$ Denote the center sphere of $N_k$ by $S_k$,
and their intersection points with $\gamma$ by  $p_{k}\in S_k\cap
\gamma$, for $k=1,2,\cdots.$

Take a sequence points  $\gamma(m)$ with $m=1,2,\cdots.$ For each
fixed neck $N_k$, arbitrarily choose a point $q_{k}\in N_k$ near
the center sphere $S_k$, draw a geodesic segment $\gamma^{km}$
from $q_{k}$ to $\gamma(m)$. Now we can show by triangle
comparison that for any fixed neck $N_l$ with $l>k$, $\gamma^{km}$
will pass through $N_l$ for all sufficiently large $m$.

 For any $s>0$, choose
two points ${\tilde{p}_{k}}$ on $\overline{p_{k}\gamma(m)}\subset
\gamma $ and ${\tilde{q}_{k}}$ on
$\overline{q_{k}\gamma(m)}\subset \gamma^{km}$ with
$d(p_{k},{\tilde{p}_{k}})=d(q_{k},{\tilde{q}_{k}})=s$. By
Toponogov comparison theorem, we have

$$
\lim\limits_{m\rightarrow
\infty}\frac{d({\tilde{p}_{k}},{\tilde{q}_{k}})}{d(p_{k},q_{k})}\geq
1.
$$

Letting $m\rightarrow \infty$, we see that $\gamma^{km}$ has a
convergent subsequence whose limit $\gamma^{k}$ is a geodesic ray
passing through all $N_l$ with $l > k$. Denote by
$p_{j}=\gamma(t_j), j=1, 2, \cdots$. From the above computation,
we deduce that
$$d(p_{k},q_{k})\leq d(\gamma(t_k+s),\gamma^{k}(s)).
$$
for all $s>0$.

Let $\varphi(x)=\lim_{t\rightarrow+\infty}(t-d(x,\gamma(t)))$ be
the Busemann function constructed from the ray $\gamma$.  By the
definition of Busemann function $\varphi$ associated to the ray
$\gamma$, we see that
$\varphi(\gamma^{k}(s_1))-\varphi(\gamma^{k}(s_2))=s_1-s_2$ for
any $s_1$, $s_2\geq0$.   Consequently, by investigating the value
of $\varphi$ on $\partial N_l$ and linearality of
$\varphi\mid_{\gamma^k},$ we know for each $l>k$, we have
$\gamma^{k}(t_l-t_k)\in \varphi^{-1}(\varphi(p_{l}))\cap N_{l}.$
This implies that the diameter of
$\varphi^{-1}(\varphi(p_{k}))\cap N_k$ is not greater the diameter
of $\varphi^{-1}(\varphi(p_{l}))\cap N_l$ for any $l>k$, which is
a contradiction as $l$ much larger than $k$. The proposition is
proved.

\end{proof}
\begin{remark} Without introducing a compact set $K,$ the
conclusion of Proposition \ref{p5.1} may not be true. The counter
examples can be given by cones with small aperture.
\end{remark}



\begin{thebibliography}{99}



\bibitem{BS1}   Brendle, S., and  Schoen, R., {\sl Classification of manifolds with weakly $1/4$-pinched
curvature}, Acta Math. 200 (2008), no. 1, 1-13.
\bibitem{BW} B\"{o}hm, C., and Wilking, B. {\sl Manifolds with positive curvature operator are space forms},
Ann. of Math. (2) 167(2008), no.3, 1079-1097.



\bibitem{CaoZ} Cao, H. D. and Zhu, X. P., {\sl A complete proof
of Poincare and geometrization conjectures--application of
Hamilton-Perelman theory of Ricci flow}, Asian J. math. {\bf 10}2
(2006), 165-492.

\bibitem{CZ05U}  Chen, B. L. and Zhu, X. P.,
{\sl Uniqueness of the Ricci flow on complete noncompact
manifolds}, J. Diff. Geom. {\bf 74} (2006), 119-154.

\bibitem{CZ05F} Chen, B. L. and Zhu, X. P.,
 {\sl Ricci flow with surgery on four-manifolds
 with positive isotropic
curvature}, J. Diff. Geom. {\bf 74} (2006), 177-264.



\bibitem{De} De Turck, D., {\sl Deforming metrics in the direction of their Ricci
tensors} J.Diff. Geom. {\bf 18} (1983), 157-162.

\bibitem{Fr} Fraser, Ailana M., {\sl Fundamental groups of manifolds
with positive isotropic curvature}, Ann. of Math. (2) {\bf 158}
(2003), no. 1, 345-354.

\bibitem{FW} Fraser, Ailana and Wolfson Jon, {\sl The fundamental
group of manifolds of positive isotropic curvature and surface
groups}, Duke Math. J. {\bf 133} (2006), no. 2, 325-334.


\bibitem{Ha82}  Hamilton, R. S., {\sl Three manifolds with positive
Ricci curvature }, J. Diff. Geom. {\bf 17} (1982), 255-306.

\bibitem{Ha86}  Hamilton, R. S., {\sl Four--manifolds with positive
curvature operator}, J. Differential Geom. {\bf 24} (1986),
153-179.





\bibitem{Ha97} Hamilton, R. S., {\sl Four manifolds with positive isotropic
curvature}, Commu. in Analysis and Geometry,{\bf 5}(1997),1-92.
(or see, {\sl Collected Papers on Ricci Flow}, Edited by H.D.Cao,
B.Chow, S.C.Chu and S.T.Yau, International Press 2002).
\bibitem{Ha} Hamilton,R.S., {\sl Three-orbifolds with positive Ricci
curvature},   521-524 {\sl Collected Papers on Ricci Flow}, Edited
by H.D.Cao, B.Chow, S.C.Chu and S.T.Yau, International Press
2002).

\bibitem{Iz} Izeki, H., {\sl Limit sets of Kleinian groups and
conformally flat Riemannian manifolds}, Invent. Math. {\bf 122}
(1995), 603-625.


\bibitem{L} Lu, P., {\sl A compactness property for solutions of the Ricci flow on orbifolds}
American Journal of Mathematics,\textbf{ 123}(2001),1103-1134.


\bibitem{Mc} Mccullough, Darryl, {\sl Isometries of elliptic
3-manifolds,} J. London Math. Soc. (2) 65(2002), no 1, 167-182.
\bibitem{MiMo} Micallef, M. and Moore, J. D., {\sl Minimal two-spheres and the
topology of manifolds with positive curvature on totally isotropic
two-planes}, Ann. of Math.(2) {\bf 127}(1988)199-227.

\bibitem{MW} Micallef, M. and Wang, M., {\sl Metrics with
nonnegative isotropic curvatures}, Duke Math. J. {\bf 72} (1993),
no. 3, 649-672.


 \bibitem{P1} Perelman, G., {\sl The entropy formula for the Ricci flow and its geometric
applications}, arXiv:math.DG/0211159.

\bibitem{P2} Perelman, G., {\sl Ricci flow with surgery on three
manifolds,}  arXiv: math. DG/0303109.


\bibitem{SY} Schoen, R. and Yau, S.T., {\sl Conformally flat
manifolds, Kleinian groups and scalar curvature}, Invent. Math. {\bf
92} (1988), 47-71.

\bibitem{S} Schoen, R., {\sl Open problems proposed in Pacific
Northwest Geometry Seminar, 2007-fall (at Univ. of Oregon)}.


\bibitem{Th} Thurston, W. {\sl Geometry and topology of three
manifolds}, Lecture notes, Princeton University, 1979.

\bibitem{Wolf}  Wolf, J., {\sl Spaces of constant curvature,}
Wilmington: Publish or Perish,1984.
\end{thebibliography}
\end{document}